\documentclass[12pt]{amsart}
\pdfoutput=1
\usepackage{amsmath,amsfonts,amssymb,amsthm}
\usepackage{pinlabel}
\usepackage{color}
\usepackage{aliascnt}
\usepackage[all,pdf]{xy}
\usepackage{graphicx}

\numberwithin{equation}{section}

\newtheorem{thm}{Theorem}[section]

\newaliascnt{corcnt}{thm}

\newaliascnt{propcnt}{thm}

\newtheorem{prop}[propcnt]{Proposition}

\newaliascnt{lemcnt}{thm}
\newtheorem{lem}[lemcnt]{Lemma}

\newaliascnt{factcnt}{thm}
\newtheorem{fact}[factcnt]{Fact}

\theoremstyle{remark}
\newaliascnt{remcnt}{thm}
\newtheorem{rem}[remcnt]{Remark}

\newaliascnt{excnt}{thm}
\newtheorem{example}[excnt]{Example}

\theoremstyle{definition}
\newaliascnt{defncnt}{thm}
\newtheorem{defn}[defncnt]{Definition}

\def\bfb{\mathbf{b}}
\def\bfa{\mathbf{a}}
\def\bfq{\mathbf{q}}
\def\bfp{\mathbf{p}}
\def\bfx{\mathbf{x}}
\def\bfy{\mathbf{y}}

\newcommand{\R}{\mathbb{R}}
\newcommand{\Q}{\mathbb{Q}}
\newcommand{\Z}{\mathbb{Z}}
\newcommand{\C}{\mathbb{C}}
\newcommand{\N}{\mathbb{N}}

\newcommand{\F}{\mathbb{F}}

\newcommand{\MC}{\mathcal{C}}
\newcommand{\A}{\mathcal{A}}

\newcommand{\MQ}{\mathcal{Q}}

\newcommand{\id}{\mathrm{id}}

\newcommand{\indx}{\operatorname{index}}

\newcommand{\co}{\mskip0.5mu\colon\thinspace}

\begin{document}

\title[Estimating the number of Reeb chords using linear
representations]{Estimating the number of Reeb chords using a linear
representation of the characteristic algebra}

\author{Georgios Dimitroglou Rizell}
\author{Roman Golovko}

\begin{abstract}
Given a chord-generic, horizontally displaceable Legendrian
submanifold $\Lambda\subset P\times \R$ with the property that its
characteristic algebra admits a finite-dimensional matrix
representation, we prove an Arnold-type lower bound for the number
of Reeb chords on $\Lambda$. This result is a generalization of the
results of Ekholm, Etnyre, Sabloff
and Sullivan, which
hold for Legendrian submanifolds whose Chekanov--Eliashberg algebras
admit augmentations. We also provide examples of Legendrian
submanifolds $\Lambda$ of $\C^{n}\times \R$, $n \ge 1$, whose
characteristic algebras admit finite-dimensional matrix
representations but whose Chekanov--Eliashberg algebras do not admit
augmentations. In addition, to show the limits of the method of
proof for the bound, we construct a Legendrian submanifold
$\Lambda\subset \C^{n}\times \R$ with the property that the
characteristic algebra of $\Lambda$ does not satisfy the rank
property. Finally, in the case when a Legendrian submanifold
$\Lambda$ has a non-acyclic Chekanov--Eliashberg algebra, using
rather elementary algebraic techniques we obtain lower bounds for
the number of Reeb chords of~$\Lambda$. These bounds are slightly
better than the number of Reeb chords  it is possible
 to achieve
with a Legendrian submanifold whose Chekanov--Eliashberg algebra is
acyclic.
\end{abstract}

\address{Department of Pure Mathematics and Mathematical
Statistics\\\newline
Centre for Mathematical Sciences\\
University of Cambridge\\\newline
Wilberforce Road\\
Cambridge CB3\,0WB\\
UK}
\email{g.dimitroglou@maths.cam.ac.uk}
\urladdr{http://www.dimitroglou.name/}

\address{Alfr\'{e}d R\'{e}nyi Institute of Mathematics\\
Hungarian Academy of Sciences\\\newline
Realtanoda u.\ 13--15\\
Budapest\\
H-1053\\
Hungary}
\email{golovko.roman@renyi.mta.hu}
\urladdr{http://sites.google.com/site/ragolovko/}

\date{\today}
\thanks{This work was partially supported by the ERC Starting Grant of Fr\'{e}d\'{e}ric Bourgeois StG-239781-ContactMath and by the ERC Advanced Grant LDTBud.}


\maketitle

\section{Introduction}
\subsection{Geometric background}
The {\em contactization of an exact symplectic $2n$--dimensional manifold}
$(P,d\theta)$ is a contact manifold $P\times \R$ equipped with the
contact structure $\xi:=\ker (dz+\theta)$, where $z$ is a coordinate
on~$\R$.
 Let $\Lambda$ be an $n$--dimensional submanifold of $P\times \R$. We
say that $\Lambda$ is a {\em Legendrian submanifold} if and only if
$T_x \Lambda\subset \xi_x$ for all $x\in \Lambda$. A smooth
$1$--parameter family of Legendrian submanifolds is called a {\em
Legendrian isotopy}. We will
always assume that $(P,d\theta)$
has finite geometry at infinity.

We note that if $\Lambda\subset P\times \R$ is a Legendrian
submanifold, then double points of the so-called {\em
Lagrangian projection} $\Pi_{L}\co P\times
\R\to P$ defined by $\Pi_{L}(x,z)=x$ correspond bijectively to integral
curves of
$\partial_z$ in $P\times \R$ having endpoints on $\Lambda$. The
vector field $\partial_{z}$ is the Reeb field of the contact form
$dz+\theta$ and its integral curves having endpoints on~$\Lambda$ are
called {\em Reeb chords on $\Lambda$}. The set of Reeb chords on $\Lambda$
will be denoted by $\MQ(\Lambda)$. We say that
$\Lambda$ is {\em chord-generic} if all self-intersections of the
Lagrangian immersion $\Pi_{L}(\Lambda)$ are transverse double
points, which, in the case when $\Lambda$ is closed, in particular implies
that $|\MQ(\Lambda)|<\infty$. Note that the chord-generic Legendrian
embeddings of a closed manifold form an open and dense subset of the space
of Legendrian embeddings.

From now on we assume that all Legendrian
submanifolds are closed, orientable, connected and chord-generic, unless
stated otherwise.

A Legendrian submanifold  $\Lambda\subset P\times \R$
is called {\em horizontally displaceable}
if the projection $\Pi_{L}(\Lambda)$
can be completely displaced off of itself by a Hamiltonian isotopy.

Observe that in the case $(P,d\theta)=(T^*M,d\theta_M)$, where $\theta_M$ is
the so-called \emph{Liouville form}, the bundle projection
$p \co
T^{\ast}\!M \to M$ induces the
natural projection $\Pi_F\co T^{\ast}\!M\times \R\to M\times \R$ defined by
$\Pi_{F}(x,z)= (p(x), z)$. This
projection is called the {\em front projection}.
As a special case, this also applies to the standard symplectic vector space
$(\C^n=T^*\R^n,-d(y_1dx_1+\cdots+y_ndx_n))$ (where we usually use minus the
Liouville~form).%

There are two ``classical'' invariants for a Legendrian submanifold $\Lambda
\subset \C^{n}\times \R$ which are invariant under Legendrian isotopy:
the {\em
Thurston--Bennequin number $\operatorname{tb}(\Lambda)$} and the {\em rotation
class $r(\Lambda)$}. The
Thurston--Bennequin number of a homologically trivial Legendrian
$\Lambda\subset P\times \R$ was first defined in the case $n=1$
by Bennequin \cite{EERDP}, and independently by Thurston, and then
extended to the case when $n\geq 1$ by Tabachnikov~\cite{AIOASTITTAD}.
 One may define it by
$\operatorname{tb}(\Lambda):=\mathrm{lk}(\Lambda,\Lambda')$, where
$\Lambda'$ is a sufficiently small push-off of~$\Lambda$ along the Reeb
vector field.
The rotation class $r(\Lambda)$ of a Legendrian submanifold
$\Lambda \subset P \times \R$ is defined as the homotopy class of the complex
bundle monomorphism $T\Lambda \otimes \C \to \xi \subset TP$ induced by the
differential of the inclusion (together with some choice of compatible almost
complex structure on $P$). We refer to Ekholm, Etnyre and Sullivan \cite{NILSIRTNPO} for more details.

{\em Legendrian contact homology} is a modern invariant of
Legendrian submanifolds in~$P\times\nobreak \R$. It is a variant of the
symplectic field theory introduced by Eliashberg, Givental and
Hofer \cite{ITSFT}. Independently, for Legendrian knots in
$\C\times \R$, it was defined by Chekanov \cite{DAOLL}. This
invariant associates a differential graded algebra (DGA) denoted by
$\A(\Lambda)$ to a Legendrian submanifold $\Lambda$, sometimes called the {\em
Chekanov--Eliashberg algebra of $\Lambda$}. $\A(\Lambda)$ is a non-commutative
unital differential graded algebra over a field $\F$ freely generated by
elements of $\MQ(\Lambda)$.  The differential $\partial(a)$ of a generator
$a \in \MQ(\Lambda)$ is given by a count of rigid pseudo-holomorphic disks
for some choice of compatible almost complex structure, and is then extended
using the Leibniz rule.

We will use the version of Legendrian contact homology for
Legendrian submanifolds of $P\times \R$, which was developed by
Ekholm, Etnyre and Sullivan \cite{LCHIPXR}. It was
shown there
that the homology of $(\A(\Lambda), \partial)$, the so-called {\em
Legendrian contact homology of~$\Lambda$}, is independent of the
choice of an almost complex structure and invariant under Legendrian
isotopy.

We now sketch the definition of the differential, as given in
\cite{LCHIPXR}. For a generic tame almost complex structure $J$ on $P$,
we define
\[
\partial(a)=
\mkern-6mu
\sum_{\dim \mathcal{M}_{a;\bfb ;A}(\Lambda;J)=0}\mkern-5mu
(-1)^{(n-1)(|a|+1)}\#(\mathcal{M}_{a;\bfb ;A}(\Lambda;J))\bfb ,
\]
where $\bfb =b_1 \cdots  b_m$ is a word of Reeb
chords. Here $\mathcal{M}_{a;\bfb ;A}(\Lambda;J)$ is the
moduli-space of $J$--holomorphic disks in $P$ having boundary on
$\Pi_L(\Lambda)$, a positive boundary puncture mapping to $a$,
negative boundary punctures mapping to $b_1,\ldots,b_m$ (in that
order relative
to the oriented boundary with the positive puncture
removed), and being in the relative homology class $A \in
H_2(P,\Pi_L(\Lambda))$. We refer to \autoref{sectrans} for the
definition of positive and negative punctures.

The {\em Maslov class} of a Legendrian submanifold $\Lambda$ is a
cohomology class $\mu(\Lambda )$
in $H^{1}(\Lambda;\Z)$ that assigns
to each $1$--dimensional homology class the Maslov index of a path
representing that class; see eg~\cite{NILSIRTNPO}. In the case when the
Maslov class vanishes, the Chekanov--Eliashberg algebra of a one-component
Legendrian submanifold has a canonical grading in $\Z$. In general,
the Chekanov--Eliashberg algebra of an oriented one-component Legendrian
submanifold has a canonical grading in $\Z_2$.

Given a Chekanov--Eliashberg algebra $(\A(\Lambda), \partial)$ over a field
$\F$, an {\em augmentation} of  $(\A(\Lambda), \partial)$
is a unital algebra chain map $\varepsilon\co(\A(\Lambda), \partial)\to (\F,
\partial=0)$, ie an algebra map satisfying
$\varepsilon(1)=1$ and $\varepsilon\circ \partial = 0$. If $\varepsilon(c)=0$
for $|c|\neq
0$, $c\in \MQ(\Lambda)$, we say that $\varepsilon$ is {\em graded}. If
$(\A(\Lambda), \partial)$ admits
an augmentation (which is not always the case), then we can follow the
linearization procedure due to Chekanov to produce a complex spanned as a
vector space by the Reeb chords. We first define a tame automorphism
$\sigma_{\varepsilon}\co (\mathcal A, \partial) \to (\mathcal A,
\partial)$ with $\sigma_{\varepsilon}(c)=c+\varepsilon(c)$. Then we
define $(\MC(\Lambda),
\partial^{\varepsilon}:=(\sigma_{\varepsilon}\circ\partial\circ\sigma_{\varepsilon}^{-1})^{}_1)$,
where $\MC(\Lambda)$ is the $\F$--vector space spanned by the elements of
$\MQ(\Lambda)$. The homology of $(\MC(\Lambda),
\partial^{\varepsilon})$ is called the {\em linearized Legendrian
contact homology of $\Lambda$}.

Observe that an augmentation can be
seen as a $1$--dimensional linear representation of $\A(\Lambda)$
which satisfies the additional property that $\varepsilon\circ\partial = 0$. A
$k$--dimensional linear representation $\rho\co \A(\Lambda)\to \mathrm
M_k(\F)$ is called  {\em graded}
if $\rho(c)=0$ for $|c|\neq\nobreak 0$,~$c\in\nobreak \MQ(\Lambda)$.
We will be
interested in such representations satisfying the additional condition
$\rho\circ\nobreak \partial =\nobreak 0$.

Assume that we are given a Legendrian submanifold $\Lambda\subset (P\times \R,
dz+\theta)$ with Chekanov--Eliashberg DGA $(\A(\Lambda),\partial)$. Ng
\cite{CLI} defined the so-called  {\em characteristic algebra
of $\Lambda$}, which is given by $\MC_{\Lambda}:=\A(\Lambda)/I$,
where $I$ denotes the two-sided ideal of $\A(\Lambda)$ generated by
$\{\partial(c)\}_{c\in \MQ(\Lambda)}$. If two
submanifolds $\Lambda$ and $\Lambda'$ are Legendrian isotopic, then
$\MC_{\Lambda}$ and $\MC_{\Lambda'}$ become isomorphic after stabilizations
by free products with suitable finitely generated free algebras, as follows
from \cite[Theorem 3.4]{CLI} together with \cite[Theorem 1.1]{LCHIPXR}.

Note that there is a one-to-one correspondence between
(graded) linear representations $\rho\co \A(\Lambda)\to \mathrm
M_{k}(\F)$ satisfying $\rho\circ \partial = 0$ and (graded) linear
representations $\MC_\Lambda\to\nobreak \mathrm M_{k}(\F)$. Again, \cite[Theorem
3.4]{CLI} together with \cite[Theorem 1.1]{LCHIPXR} shows the following
important invariance result:

\begin{prop}
\label{propinvariance}
The property of having a characteristic algebra admitting a (graded)
$k$--dimensional representation is invariant under Legendrian isotopy and
independent of the choice of almost complex structure used in the definition
of the Chekanov--Eliashberg algebra.
\end{prop}

The {\em front $S^1$--spinning} is a procedure defined by Ekholm,
Etnyre and Sullivan \cite{NILSIRTNPO} which, given a Legendrian
submanifold $\Lambda\subset \C^{n}\times \R$, produces a Legendrian
submanifold $\Sigma_{S^1}\Lambda\subset \C^{n+1}\times \R$
diffeomorphic to $\Lambda\times S^1$.  The second
author \cite{ANOTFSC} generalized this construction to a notion of {\em the front
$S^m$--spinning} for $m\in \mathbb N$. Given a Legendrian submanifold
$\Lambda \subset \C^{n}\times \R$, this construction produces a
Legendrian submanifold $\Sigma_{S^m}\Lambda\subset \C^{n+m}\times
\R$ diffeomorphic to $\Lambda\times S^m$. Observe that the front
$S^m$--spinning construction can be seen as a particular case of the
{\em Legendrian product} construction defined by Lambert-Cole
~\cite{LP}.

\subsection{Algebraic background}
In this section, we recall several basic definitions from non-commutative ring
theory that will become useful later.

\begin{defn}
A
ring $R$ satisfies
\begin{itemize}
\item
{\em the invariant basis number (IBN) property}  if the left-modules
$R^n$ and $R^m$ are isomorphic if and only if $m=n$;
\item {\em the rank property} if from the existence
of an epimorphism $R^n\to R^m$ of free left modules it follows that
$n\geq m$;
\item {\em the strong rank property} if from the
existence of monomorphism $R^n\to R^m$ of free left modules it
follows that $n\leq m$.
\end{itemize}
\end{defn}

The
following three remarks, which describe basic
properties of these three conditions,
can all be derived using
elementary algebra.
\vadjust{\goodbreak}

\begin{rem}\label{rhomoibnrcondpres}
Let $f\co R\to S$ be a (unital) ring homomorphism. If $S$ satisfies the rank
property (the IBN
property), then $R$ satisfies the
rank property (the IBN property). Unfortunately, the same principle does
not hold for the strong rank property in general.
\end{rem}

\begin{rem}\label{genpictimplibnrpstrrkpr}
For
a unital ring we have the implications
$$
\mbox{strong rank property}\ \implies \ \mbox{rank
property}\ \implies \ \mbox{IBN property}.
$$
In general, none of the reversed implications are true.
\end{rem}

\begin{rem}\label{nonibnpnontrp}
A
ring $R$ fails to satisfy the rank property if and
only if there exist natural numbers $m>n$ and matrices $A$, $B$ over
$R$ of sizes $m\times n$ and $n\times m$, respectively, such that
$AB=I_m$.
\end{rem}

\subsection{Results}
Many important techniques in symplectic and contact topology are
designed to capture quantitative properties of dynamical systems
induced by, say, a Hamiltonian flow or a Reeb flow. The problem of
finding closed orbits of a Hamiltonian vector field and a Reeb
vector field can in turn be related to the problems of finding
intersections between Lagrangian submanifolds and Reeb chords on
Legendrian submanifolds, respectively. Several examples of rigidity phenomena
have been established in different settings, where the numbers of such
intersection points, or chords, are strictly greater than the numbers
predicted by topology alone.
Often
 these estimates are proven using
different variants of Floer homology.

\subsubsection{The first lower bounds}
Gromov's theorem \cite[Theorem 2.3.$\mathrm{B}_1$]{PCISM}
 implies that any
horizontally displaceable Legendrian submanifold must have at least
one Reeb chord. Note that this estimate is already
 highly
non-trivial in general.

In even dimensions, there is a lower bound obtained from topology
alone. Assume that we are given a chord-generic, horizontally
displaceable, orientable Legendrian submanifold $\Lambda\subset
P\times \R$, where $\dim P=2n$ and $n=2k$ is even. The Whitney
self-intersection index of $\Pi_{L}(\Lambda)$ is equal to
$(-1)^{k+1} \frac{1}{2} \chi(\Lambda)$, as follows from purely
topological considerations; see Audin \cite[Proposition
0.4]{FNIEDDPDILEPTR} for $P=\C^n$. In particular, $\Pi_{L}(\Lambda)$
has at least
$\frac{1}{2}|\chi(\Lambda)|$ double points,
and hence $|\MQ(\Lambda)|\geq \frac{1}{2}|\chi(\Lambda)|$.
In the case when $P=\C^{2k}$, \cite[Proposition
3.2(2)]{NILSIRTNPO}
shows moreover that
$\operatorname{tb}(\Lambda)=(-1)^{k+1} \frac{1}{2} \chi(\Lambda)$.

\subsubsection{Arnold's inequality}
For some time it was expected that a horizontally displaceable,
chord-generic, Legendrian submanifold would have at least as many Reeb
chords as half the sum of its Betti numbers. This type of inequality
is in line with other conjectures due to Arnold \cite{FSIST}, and it is
usually referred to as \emph{Arnold's inequality} or an \emph{Arnold-type
inequality}. This inequality is however not fulfilled in general; we refer
to \autoref{secfailure} for references to counter-examples belonging to
the so-called flexible side of contact geometry.

On the other hand, it has been shown that the inequality is
fulfilled for certain natural Legendrian isotopy classes. In the
case when the Chekanov--Eliashberg algebra of $\Lambda\subset P\times
\R$ admits an augmentation, the inequality was originally proven by
Ekholm, Etnyre and Sullivan \cite{OILCHAELI}. A ``graded''
refinement of this result is also a consequence of the (considerably
stronger) duality result for linearized Legendrian contact homology
due to Ekholm, Etnyre and Sabloff \cite{ADESFLCH}. In this paper, we show
that the requirement of having an augmentation can be relaxed considerably
in the proof of the above
inequality.

\begin{thm}\label{mainthartinchlocbig}
Let $\Lambda\subset P\times \R$ be an $n$--dimensional horizontally
displaceable Legendrian submanifold which is chord-generic. If its
characteristic algebra $\MC_{\Lambda}$ admits a $k$--dimensional
representation $\rho\co \MC_{\Lambda}\to \mathrm{M}_{k}(\F)$ for some
field $\F$, then the following inequality holds:
\begin{equation}\label{weekarniineqbcfgagbn}
\frac{1}{2} \sum_{i\in I} b_i \le \sum_{i\in I} c_i,
\end{equation}
where $b_j:=\dim_{\F} H_j (\Lambda;\F)$, $c_j$ is
the number of Reeb
chords on $\Lambda$ of grading $j$ and $I\in\{ 2\Z,\, \Z\!\setminus\!
2\Z\}$.

In addition, if $\rho$ is a graded representation, then
the following refinement of the previous inequality holds:
\begin{equation}\label{arnineqmrepbnbad}
b_{i} \leq c_{i} + c_{n-i}
\end{equation}
for all $i$ satisfying $0\leq i\leq n$.
\end{thm}

\begin{rem} In the case when $\operatorname{char} \F \neq 2$, we must
make the additional assumption that $\Lambda$ is spin in order for its
Chekanov--Eliashberg algebra to be well-defined with coefficients in $\F$.
\end{rem}

To provide interesting examples of Legendrian submanifolds
satisfying the assumptions of the theorem we will proceed as follows. In the
literature, there are many computations of Chekanov--Eliashberg
algebras of Legendrian knots in $\C\times \R$. We provide
higher-dimensional examples by using the front $S^m$--spinning construction.
The Chekanov--Eliashberg algebra of the $S^1$--spinning
of a
$1$--dimensional Legendrian knot was computed in terms of the
Chekanov--Eliashberg algebra of the original knot
by Ekholm and K\'{a}lm\'{a}n~\cite{IOL1KAL2T}. In \autoref{techsectspuninvhom},
we provide a part of this computation in the general case. Even
though we do not compute the full Chekanov--Eliashberg algebra of the
$S^m$--spinning of a general Legendrian submanifold $\Lambda$, we are
able to establish the following result:

\begin{thm}\label{cormrftspgdsy}
A Legendrian submanifold $\Lambda \subset \C^n\times\R$ and its
$S^m$--spinning
$\Sigma_{S^m}\Lambda \subset \C^{n+m}\times\R$
satisfy the
following relations:
\begin{enumerate}
\item The Chekanov--Eliashberg algebra of $\Sigma_{S^m}\Lambda$ admits a
(graded) augmentation in a unital ring $R$ if and only if that
of $\Lambda$ admits a (graded) augmentation in $R$.
\item The characteristic algebra of $\Sigma_{S^m}\Lambda$ admits a
(graded) $k$--dimensional representation
if and only if that of $\Lambda$ admits a (graded) $k$--dimensional
representation.
\item The Chekanov--Eliashberg algebra of $\Sigma_{S^m}\Lambda$ is
acyclic if and only if that of $\Lambda$ is acyclic.
\end{enumerate}
\end{thm}

Sivek \cite{TCHOLKWMTBI} has provided examples of Legendrian
knots in $\C\times \R$ whose characteristic algebras admit
$2$--dimensional,
 but not $1$--dimensional, representations over
$\Z_2$. In \autoref{exfindimreprbnaug}, we apply the front
$S^m$--spinning construction to these knots and other knots
constructed from them and, using \autoref{cormrftspgdsy}, we prove
the following:

\begin{thm}\label{applsptwtherffdrbnaugalk}
Fix a product $S:=S^{m_1}\times \dots \times S^{m_s}$ of spheres.
There exists an infinite family of Legendrian embeddings $\Lambda_i
\subset \C^{1+m_1+\cdots+m_s} \times \R$, $i=1,2,\ldots{}$, of $S^1
\times S$ which are in pairwise different Legendrian isotopy classes
and whose characteristic algebras all admit finite-dimensional
matrix representations over $\Z_2$, but whose Chekanov--Eliashberg
algebras do not admit augmentations to unital commutative rings.
In particular, \autoref{mainthartinchlocbig} can be applied to
these Legendrian submanifolds.
\end{thm}

Observe that the condition that there exists a finite-dimensional
representation of $\mathcal C_{\Lambda}$ over a field $\F$ is a
natural condition that can be used to define the notion of rank for
free $\MC_{\Lambda}$--$\MC_{\Lambda'}$--bimodules in the case when
$\Lambda'$ is Legendrian isotopic to $\Lambda$. It also allows us to
write rank inequalities for long exact sequences of such free
finite-rank $\MC_{\Lambda}$--$\MC_{\Lambda'}$--bimodules. These rank
inequalities play a crucial role in the proof of
\autoref{mainthartinchlocbig}. We do not know if this
restriction on $\MC_{\Lambda}$ can be weakened. On the other hand,
in \autoref{arglimsecttomonbg} we show that there are
examples of Legendrian submanifolds of $\C^{n}\times \R$ with a
``wild'' behavior of $\MC_{\Lambda}$. More precisely, there are
$n$--dimensional Legendrian submanifolds~$\Lambda$ with non-trivial
$\MC_\Lambda$ for which there exists a monomorphism (epimorphism)
from a free $\MC_{\Lambda}$--$\MC_{\Lambda'}$--bimodule of rank $k$ to a
free $\MC_{\Lambda}$--$\MC_{\Lambda'}$--bimodule of rank $l$, where $k>
l$~($k< l$). In other words, we prove the following:

\begin{thm}\label{hdbexwithoutthestrrnkpr}
Given any $n\geq 1$, there exists a Legendrian submanifold $\Lambda\subset
\C^n\times\R$
whose characteristic algebra $\MC_{\Lambda} \neq 0$ does not satisfy the
rank property.
\end{thm}

This
result shows that it is unlikely that the method in the proof of
\autoref{mainthartinchlocbig} can be used to prove an Arnold-type inequality
for a general horizontally displaceable Legendrian submanifold having a
Chekanov--Eliashberg algebra which is not acyclic.

\subsubsection{A failure of Arnold's inequality}
\label{secfailure}
There are examples of horizontally displaceable Legendrian submanifolds
$\Lambda \subset P \times \R$ for which Arnold's inequality is not
satisfied. Observe that all known examples have an acyclic Chekanov--Eliashberg
algebra.

The first example of a Legendrian submanifold for which Arnold's
inequality does not hold was provided by Sauvaget \cite{CLED4},
who constructed a genus-two Legendrian surface in $\C^2\times
\R$ having only one transverse Reeb chord.

Loose Legendrian submanifolds is a class of Legendrian submanifolds
defined by Murphy \cite{LLEIHDCM},
who showed they
satisfy an
h-principle. Since these submanifolds belong to
the flexible domain of contact geometry, one does not expect them to satisfy
any rigidity phenomena. Observe that the
Chekanov--Eliasbherg algebra of a loose Legendrian submanifold is acyclic
(with or without Novikov coefficients).
Using this h-principle, together
with the h-principle for exact Lagrangian caps as shown
in~\cite{LC}, Ekholm, Eliashberg, Murphy and Smith
\cite{CELIWFDP} provided many examples of exact Lagrangian
immersions with few double points. We
present here a weaker form of
their result.

\begin{thm}[\cite{CELIWFDP}]
Suppose that $\Lambda$ is a smooth closed $n$--dimensional manifold for
which $T \Lambda \otimes \C$ is a trivial complex bundle. There exists a
loose, horizontally displaceable, chord-generic, Legendrian embedding $\Lambda
\subset \C^n \times \R$ satisfying
\[
\begin{cases}
1 \le |\MQ(\Lambda)| \le 2& \text{if }\, n \,\text{ is odd},\\
|\MQ(\Lambda)|=\frac{1}{2}|\chi(\Lambda)| & \text{if }\, n \,\text{ is even
and }\, \chi(\Lambda)<0,\\
\frac{1}{2}|\chi(\Lambda)| \le |\MQ(\Lambda)| \le
\frac{1}{2}|\chi(\Lambda)|+2 & \text{if }\, n \,\text{ is even and }\,
\chi(\Lambda)>0.
\end{cases}
\]
\end{thm}

Some rather basic algebraic considerations show the following
(very) slight improvement of the lower bound in the case when a Legendrian
submanifold has a non-acyclic Chekanov--Eliashberg algebra.
\begin{prop}\label{aprbpundalgprfbwhkn}
Suppose that $\Lambda \subset P \times \R$ is a horizontally
displaceable, chord-generic $n$--dimensional Legendrian submanifold whose
characteristic algebra is non-trivial, but does not
admit any finite-dimensional representations. It follows that
\[
|\MQ(\Lambda)|\geq 3.
\]
Moreover, if $n=2k$, we have the bound
\[
|\MQ(\Lambda)|\geq \tfrac{1}{2}|\chi(\Lambda)|+2
\]
under the additional assumptions that $\Lambda$ is orientable and either
\begin{enumerate}
\item $\chi(\Lambda) \ge 0$, or
\item $\mu(\Lambda)=0$ and all generators have non-negative grading.
\end{enumerate}
\end{prop}

\subsection*{Acknowledgements}
This work was partially supported by the ERC Starting Grant of
Fr\'{e}d\'{e}ric Bourgeois StG-239781-ContactMath and by the ERC Advanced
Grant LDTBud.
The authors are deeply grateful to Yakov Eliashberg, John Etnyre and
Michael Sullivan for helpful conversations and interest in their
work. In addition, Golovko would like to thank
the Simons
Center, where part of this article was written, for its hospitality. Finally,
the authors are grateful to the referee for many valuable comments and
suggestions.

\section{Proof of \texorpdfstring{\autoref{mainthartinchlocbig}}{Theorem 1.6}}

The proof uses the same construction and idea as Ekholm, Etnyre and Sabloff
\cite{ADESFLCH}. In
other words, we consider the two-copy link $\Lambda \cup \Lambda'$
constructed as follows. First, identify a neighborhood of $\Lambda$ with
a neighborhood of the zero-section $\Lambda \subset J^1(\Lambda)$ using
the standard neighborhood theorem for Legendrian submanifolds. Using this
identification, the section $j^1f \subset J^1(\Lambda)$ can be considered
as a Legendrian submanifold $\Lambda '' \subset P \times \R$, where $f
\co \Lambda \to \R$ is a $C^2$--small Morse function. $\Lambda'$ is now
obtained from $\Lambda''$ by a translation sufficiently far in the positive
Reeb direction, so that there are no Reeb chords starting on $\Lambda'$ and
ending on $\Lambda$. Observe that $\Pi_L(\Lambda')$
may still be assumed
to be arbitrarily close to $\Pi_L(\Lambda)$ in $C^1$--norm, and hence that
there is a canonical bijective correspondence $\mathcal{Q}(\Lambda)\simeq
\mathcal{Q}(\Lambda')$ of Reeb chords.

By topological considerations of the disks involved in the
Chekanov--Eliashberg algebra of $\Lambda \cup \Lambda'$, we see that there
is a filtration
\[
(\mathcal{A}(\Lambda \cup \Lambda'),\partial) \supset \cdots \supset
(\mathcal{A}(\Lambda \cup \Lambda')^1,\partial) \supset (\mathcal{A}(\Lambda
\cup \Lambda')^0,\partial),
\]
where $(\mathcal{A}(\Lambda \cup \Lambda')^i,\partial)$ is spanned by
words containing at most $i$ letters corresponding to mixed chords having
starting point on $\Lambda$ and endpoint on $\Lambda'$. We can thus define
the quotient complex
\[
(\mathcal{A}(\Lambda,\Lambda'),\partial) := (\mathcal{A}(\Lambda \cup
\Lambda')^1,\partial)/\mathcal{A}(\Lambda \cup \Lambda')^0,
\]
which we identify with the  vector space over $\F$ spanned by words of
the form
$\bfa c\bfb $, where $\bfa $ is a word of Reeb chords
on $\Lambda$, $c$ is a Reeb chord starting on $\Lambda$ and ending
on $\Lambda'$, and $\bfb $ is a word of Reeb chords on
$\Lambda'$. In fact, this $\F$--vector space is naturally a free
$\mathcal{A}(\Lambda)$--$\mathcal{A}(\Lambda')$--bimodule spanned by the
set $\mathcal{Q}(\Lambda,\Lambda')$ of Reeb chords starting on $\Lambda$
and ending on $\Lambda'$. However, the differential is in general not a
morphism of bimodules.

From the existence of  $\rho\co \mathcal{C}_\Lambda \to \mathrm{M}_k(\F)$
and the fact that $\Lambda$ and $\Lambda'$ are Legendrian
isotopic, it follows from \autoref{propinvariance} that there
exists a linear representation $\rho'\co\mathcal{C}_{\Lambda'} \to
\mathrm{M}_k(\F)$ as well. Let $\mathcal{C}(\Lambda,\Lambda')$ denote
the free $\mathcal{C}_\Lambda$--$\mathcal{C}_{\Lambda'}$--bimodule
generated by $\mathcal{Q}(\Lambda,\Lambda')$. Similarly,
let $\mathcal{C}_\rho(\Lambda,\Lambda')$ denote the free
$\mathrm{M}_k(\F)$--$\mathrm{M}_k(\F)$--bimodule generated by
$\mathcal{Q}(\Lambda,\Lambda')$. Observe that the compositions
\begin{gather*}
\mathcal{A}(\Lambda) \xrightarrow{\pi_\Lambda} \mathcal{C}_\Lambda \xrightarrow{\rho}\mathrm{M}_k(\F),\\
\mathcal{A}(\Lambda') \xrightarrow{\pi_{\Lambda'}} \mathcal{C}_{\Lambda'} \xrightarrow{\rho'} \mathrm{M}_k(\F),
\end{gather*}
where $\pi_\Lambda$ and $\pi_{\Lambda'}$ are the quotient projections,
induce a commutative diagram
\[
\xymatrix{\mathcal{A}(\Lambda,\Lambda') \ar@{->>}[r]
\ar[d]_{\partial} & \mathcal{C}(\Lambda,\Lambda') \ar[r] \ar[d] &
\mathcal{C}_\rho(\Lambda,\Lambda') \ar@{-->}[d] \\
\mathcal{A}(\Lambda,\Lambda') \ar@{->>}[r] & \mathcal{C}(\Lambda,\Lambda')
\ar[r] & \mathcal{C}_\rho(\Lambda,\Lambda'),}
\]
where all maps are $\F$--linear and the horizontal maps are the
induced natural $\mathcal{A}(\Lambda)$--$\mathcal{A}(\Lambda')$--bimodule
morphisms. The middle vertical map is a uniquely determined morphism of
$\mathcal{C}_\Lambda$--$\mathcal{C}_{\Lambda'}$--bimodules induced by the
differential $\partial$, as follows from the Leibniz rule
\[
\partial(\bfa c\bfb )=\partial_{\Lambda}(\bfa )c\bfb +
(-1)^{|\bfa c|}
\bfa c\partial_{\Lambda'}(\bfb )+(-1)^{|\bfa |}\bfa \partial(c)\bfb
\]
together with the fact that $\partial_{\Lambda}(\bfa )$ and
$\partial_{\Lambda'}(\bfb )$ are in the kernels of the
projection to the respective characteristic algebras. It follows that the
rightmost vertical map is uniquely determined by the requirement of being
an $\mathrm{M}_k(\F)$--$\mathrm{M}_k(\F)$--bimodule morphism making the diagram
commutative.

\begin{rem}
In the case when $k=1$, the rightmost vertical map is simply the
linearized differential with respect to the augmentation $\epsilon \co
\mathcal{A}(\Lambda \cup \Lambda') \to \F$
that is defined as follows:
$\epsilon$ vanishes on generators corresponding to chords starting on
$\Lambda$ and ending on $\Lambda'$, while it takes the value $\rho \circ
\pi_{\Lambda}$ and $\rho' \circ \pi_{\Lambda'}$ on generators corresponding
to chords on $\Lambda$ and $\Lambda'$, respectively.
\vadjust{\goodbreak}%
\end{rem}

The assumption of horizontal displaceability implies
that there is
a Legendrian isotopy $\Lambda_t \cup \Lambda'_t$ for which $\Lambda_0
\cup \Lambda'_0=\Lambda \cup \Lambda'$ and such that there are no
chords between $\Lambda_1$ and $\Lambda'_1$. The invariance proof of
Ekholm, Etnyre and Sullivan
\cite{LCHIPXR} implies that, after a finite number of stabilizations
by the direct sum of a trivial complex with two generators, the complex
$(\mathcal{A}(\Lambda,\Lambda'),\partial)$ is isomorphic to the stabilization
of a trivial complex. It follows that the same is true for the complex
$(\mathcal{C}_\rho(\Lambda,\Lambda'),\partial)$. In particular, the latter
complex is acyclic as well.

Observe that there is a natural isomorphism
\[
\Phi\co\mathrm{M}_k(\F) \otimes \mathrm{M}_k(\F)^{\mathrm{\mathrm{op}}}
\to \mathrm{M}_{k^2}(\F)
\]
of $\F$--algebras determined by taking the values $\Phi(A\otimes B)(v\otimes
w):=A(v)\otimes B^T(w)$ and then extended by linearity
for $v, w \in
\F^k$
and $A, B\in \mathrm{M}_k(\F)$,
where we have used an identification
$\F^k \otimes \F^k \simeq \F^{k^2}$. The considerations of the set
$\mathcal{Q}(\Lambda,\Lambda')$
in \cite[Section~3.1]{ADESFLCH} show that
\[
\mathcal{C}_\rho(\Lambda,\Lambda')=\mathcal{Q} \oplus \mathcal{C} \oplus
\mathcal{P},
\vspace{-5pt}
\]
where
$$
\mathcal{Q}=\mathrm{M}_{k^2}(\F)^{\oplus{\mathcal{Q}(\Lambda)}},
\quad
\mathcal{C}=\mathrm{M}_{k^2}(\F)^{\oplus{\mathrm{Crit}(f)}},
\quad
\mathcal{P}=\mathrm{M}_{k^2}(\F)^{\oplus{\mathcal{Q}(\Lambda)}}.
$$
Here, $\mathcal{Q}$ and $\mathcal{P}$ are both generated by subsets of
$\mathcal{Q}(\Lambda,\Lambda')$ which are in canonical bijection with
$\mathcal{Q}(\Lambda)$, and $\mathcal{C}$ is generated by Reeb chords which
are in canonical bijection with the critical points of $f$.

Moreover, using appropriate choices of Maslov potentials, a generator $q_c
\in \mathcal{Q}$ corresponding to the chord $c \in \mathcal{Q}(\Lambda)$ is
graded by $|q_c|=|c|$, a generator $c_x \in \mathcal{C}$ corresponding to the
critical point $x$ is graded by $|c_x|=\mathrm{index}_{\mathrm{Morse}}(x)-1$,
while a generator $q_c \in \mathcal{P}$ corresponding to the chord $c \in
\mathcal{Q}(\Lambda)$ is graded by $|q_c|=-|c|+n-2$. Finally, we may assume
that the actions of the generators in $\mathcal{Q}$ are strictly greater
than the actions of the generators in $\mathcal{C}$, which, in turn, are
strictly greater than the actions of the generators in $\mathcal{P}$.

The analysis in \cite[Theorem 3.6]{ADESFLCH} moreover shows that, for a
suitable choice of almost complex structure and metric on $\Lambda$, the
differential with respect to the above decomposition is of the form
\[
\partial=\begin{pmatrix} * & 0 & 0 \\
* & \partial_f & 0 \\
* & * & *
\end{pmatrix}\!,
\]
where $\partial_f$ is the Morse differential induced by a Morse--Smale pair
$(f,g)$ (where the coefficients have been taken in $\mathrm{M}_{k^2}(\F)$).
\vadjust{\goodbreak}%

The inclusion of subcomplexes
\[
\mathcal{P} \subset \mathcal{C} \oplus \mathcal{P}\subset
(\mathcal{C}_\rho(\Lambda,\Lambda'),\partial)
\]
induces the long exact sequence
\[
\cdots \to H_\bullet(\mathcal{P}) \to H_\bullet(\mathcal{C} \oplus
\mathcal{P}) \to H^{\mathrm{Morse}}_{\bullet+1}(f;\mathrm{M}_{k^2}(\F))
\to H_{\bullet-1}(\mathcal{P}) \to \cdots
\]
in homology. Since the acyclicity of
$(\mathcal{C}_\rho(\Lambda,\Lambda'),\partial)$ implies that
\[
H_\bullet(\mathcal{C} \oplus \mathcal{P}) \simeq H_{\bullet+1}(\mathcal{Q}),
\]
we obtain the long exact sequence
\[
\cdots \to H_\bullet(\mathcal{P}) \to H_{\bullet+1}(\mathcal{Q})
\to H^{\mathrm{Morse}}_{\bullet+1}(f;\mathrm{M}_{k^2}(\F)) \to
H_{\bullet-1}(\mathcal{P}) \to \cdots.
\]
Exactness and the fact that all modules in the long exact sequence
are finite-dimensional $\F$--vector spaces imply that we have the
inequality
\begin{equation}\label{mineqanofthernulgtr}
\dim_\F H^{\mathrm{Morse}}_i(f;\mathrm{M}_{k^2}(\F)) \le \dim_\F
H_{i-2}(\mathcal{P}) + \dim_\F H_i(\mathcal{Q}),
\end{equation}
where we can compute the left-hand side to be
\[
\dim_\F H^{\mathrm{Morse}}_i(f;\mathrm{M}_{k^2}(\F))=k^4 \dim_\F
H^{\mathrm{Morse}}_i(f;\F)=k^4b_i,
\]
since we are using field coefficients and
$\dim_\F
\mathrm{M}_{k^2}(\F)=k^4$. On the other hand, we also have the bounds
\begin{gather*}
\dim_\F H_i(\mathcal{Q})\le \dim_\F \mathcal{Q}_i =  k^4 c_i,\\
\dim_\F H_{i-2}(\mathcal{P})\le \dim_\F \mathcal{P}_{i-2} = k^4 c_{n-i},
\end{gather*}
from which the sought inequality follows.

\section[A partial computation of the Chekanov--Eliashberg algebra for
an Sm spun Legendrian]
{A partial computation of the Chekanov--Eliashberg\\
algebra for an $S^m$--spun Legendrian}
\label{techsectspuninvhom}

In
 the following we let $\Lambda \subset P \times \R$ be an
arbitrary closed chord-generic Legendrian submanifold of the contactization
of a Liouville domain. We
prove here
the following relationship between the Chekanov--Eliashberg algebra
of the Legendrian submanifold $S^m \times \Lambda \subset T^*S^m \times P
\times \R$ and the
Chekanov--Eliashberg algebra of $\Lambda$:

\begin{thm}\label{maininclinviig} For a suitable almost complex structure
on $T^*S^m \times P$ and a suitable Legendrian chord-generic perturbation
$L$ of $S^m \times \Lambda \subset T^*S^m \times P \times \R$, there is a
canonical inclusion
\[
\iota \co (\mathcal{A}_\bullet(\Lambda),\partial_\Lambda) \to
(\mathcal{A}_\bullet(L),\partial_L)
\vadjust{\goodbreak}
\]
of unital DGAs which, moreover, can be left-inverted by a unital DGA
morphism
\[
\pi \co (\mathcal{A}_\bullet(L),\partial_L) \to
(\mathcal{A}_\bullet(\Lambda),\partial_\Lambda).
\]
If $P=\C^{n}$, then the same is true for the corresponding Legendrian
submanifold obtained by an inclusion induced by an exact symplectic embedding
$$
(T^*S^m \times T\R^n,
d\theta_{S^m}
\oplus
d\theta_{\R^n})
\hookrightarrow (T\R^{n+m},d\theta_{\R^{n+m}}),
$$
 ie the $S^m$--spinning
$\Sigma_{S^m}\Lambda \subset J^1\R^{n+m}$ of $\Lambda$.
\end{thm}
\begin{rem}
Ekholm and K\'alm\'an
  \cite{IOL1KAL2T} computed the \emph{full}
 Chekanov--Eliashberg algebra of
$\Sigma_{S^1}\Lambda$ in terms of
the Chekanov--Eliashberg algebra
of $\Lambda$
under the additional assumption that $\dim \Lambda=1$.
\end{rem}

Note that \autoref{cormrftspgdsy} immediately follows from
\autoref{maininclinviig} using elementary algebraic considerations.

We postpone the proof of \autoref{maininclinviig} to \autoref{proof}. The
idea is to find a geometric correspondence between the pseudo-holomorphic
disks defining the differential $\partial_\Lambda$ and those defining
$\partial_L$. This correspondence will be induced by the canonical projection
map $T^*S^m \times P \to P$ for some sufficiently symmetric perturbation
of $L$ and choice of almost complex structure. Even though not all disks in
the definition of $\partial_L$ are transversely cut out for these choices,
we will show that the parts of the differential needed to deduce the above
result still will be determined by transversely cut-out disks.

\subsection{Constructing the perturbation $L$ of $S^m \times \Lambda$}
\label{secpert}

Consider the two antipodal points
$$
N=(0,\ldots,0,1) \:\: \text{and} \:\:
S=(0,\ldots,0,-1)
$$
on the unit sphere $S^m \subset
\R^{m+1}$. We will take $g \co S^m \to\nobreak [0,1]$ to be the
Morse function with two critical points obtained as the restriction of
$\frac{1}{2}(1+x_{m+1})$, where $x_{m+1}$ denotes the last standard coordinate
on~$\R^{m+1}$.
The critical points of $g$ are obviously $S$ and $N$ with critical
values $g(S)=0$ and~$g(N)=1$, respectively. Moreover, we have $g \circ r =g$
for any $r \in O(m) \subset O(m+1)$, where we use $O(m)$ to denote
the orthogonal transformations of $\R^{m+1}$ that fix $\{N,S\}$ pointwise.

In the following we will let $h \co \Lambda \to \R$ denote the
$z$--coordinate restricted to $\Lambda \subset P \times \R$. In other words,
$-h$ is a primitive of $\theta$ pulled back to $\Lambda$. Consider the
so-called Liouville flow $\phi^t \co P \to P$ with
respect to $\theta$, which is determined by the property that
$(\phi^t)^*(\theta)=e^t \theta$. One constructs a Legendrian isotopy $L_t
\subset T^*S^m \times P$ parametrized~by%
\begin{gather*}
S^m \times \Lambda \to T^*S^m \times P \times \R,\\
(\mathbf{q},x) \mapsto (-h(x) d\,e^{t\,g(\mathbf{q})},\phi^{t\,g(\mathbf{q})}(x),h(x)e^{t\,g(\mathbf{q})}),
\end{gather*}
for which $L_0 = S^m \times \Lambda$. We will be interested in the Legendrian
embedding $L:=L_\epsilon$ for some sufficiently small $\epsilon>0$, which
thus may be assumed to be arbitrarily $C^1$--close to the Legendrian embedding
$S^m \times \Lambda$.

Let
\[
\pi_{T^*S^m} \co T^*S^m \times P \to T^*S^m
\]
denote the canonical projection. Under the assumption that
$\Lambda$ is chord-generic, it follows that $L$ is chord-generic as
well. The Reeb chords on $L$ are all contained in either
$\pi_{T^*S^m}^{-1}(S)$ or  $\pi_{T^*S^m}^{-1}(N)$, and we use $\mathcal{Q}_S$
and
$\mathcal{Q}_N$ to denote the respective sets of Reeb chords.
Observe that there is a canonical bijection between each of these
subsets of Reeb chords on $L$ with the Reeb chords on $\Lambda$. More
precisely, one can say the following:
\begin{lem}
\label{lembijection} The canonical bijection
$\mathcal{Q}_S \simeq \mathcal{Q}(\Lambda)$ preserves both the action and
index of the chords
while, for the canonical bijection $\mathcal{Q}_N \simeq
\mathcal{Q}(\Lambda)$, the action is multiplied by $e^\epsilon$ and the
grading is increased by $m$.
\end{lem}

\subsection{A suitable almost complex structure}
\label{secacs}

Here we construct an almost complex structure $J$ on $T^*S^m \times P$
that suits our~needs.%

\subsubsection{An integrable almost complex structure on $T^*S^m$}

Let $i$ denote the standard complex structure on
\[
A^m_1=\{ z_1^2 +\cdots+z_{m+1}^2=1\} \subset \C^{m+1}.
\]
We will use the following explicit identification between $T^*S^m$ and
$A^m_1$. Using the round metric together with the canonical embedding
$S^m \subset \R^{m+1}$, we identify $(\bfq ,\bfp )\in T^*S^m$
with the tangent vector $v_{\bfp } \in T_{\bfq }S^m \subset
T_{\bfq }\R^{m+1}$. We now consider the identification
\begin{gather*} \Psi \colon T^*S^m \to A^m_1 \subset \C^{m+1} \\
(\mathbf{q},\mathbf{p}) \mapsto \sqrt{1+\|v_{\mathbf{p}}\|^2}\mathbf{q}-iv_{\mathbf{p}}.
\end{gather*}
It is readily checked that
$\eta(\bfx +i\bfy ):=\sqrt{1+\|\bfy \|^2}$ is strictly
plurisubharmonic on $(\C^{m+1},i)$, ie that $-d(d\eta \circ i)$ is a
K\"ahler form, and that moreover
\[
\Psi^*(-d\eta \circ i)=\theta_{S^m},
\]
where $\theta_{S^m}$ is the Liouville form on $T^*S^m$.

Observe that the linear action of $r\in O(m+1)$ on $\R^{m+1}$, which
preserves $S^m$,
extends to a complex linear action by $r\in U(m+1)$ on
$\C^{m+1}$ which preserves both $A^m_1$ and~$A^m_1 \cap
\mathfrak{Re}(\C^{m+1})$. Since this action moreover preserves $\eta$,
it is a K\"ahler isometry of $A^m_1$. Finally, it can be checked that
the corresponding action on $T^*S^m$ induced by the identification $\Psi$
coincides with the canonical symplectomorphism $r^*$.

In conclusion, using $i$ to denote the almost complex structure on $T^*S^m$
induced by the above identification $\Psi$, we have shown:

\begin{lem}
$(T^*S^m,i,d\theta_{S^m})$ is a K\"ahler manifold for which each
\[
r^* \co (T^*S^m,i,d\theta_{S^m}) \to (T^*S^m,i,d\theta_{S^m}),\quad r \in O(m+1),
\]
is a K\"ahler isometry.
\end{lem}

\subsubsection{A tame almost complex structure on $T^*S^m \times P$}

Let $J_P$ be a fixed compatible almost complex structure
on $(P,d\theta)$. For simplicity we will assume that, outside of a compact
set, $(P,d\theta)$ is exact symplectomorphic to half a symplectization,
where $J_P$ moreover is cylindrical. In particular, the latter condition
implies that the Liouville flow $\phi^t \co P \to P$ is a biholomorphism
outside of some compact set.

For $\epsilon>0$ sufficiently small, there is a
tame almost complex structure $J$ on $T^*S^m \times P$ determined by
the requirement that the diffeomorphism
\begin{gather*}
\Phi \co (T^*S^m \times P,i \oplus J_P) \to (T^*S^m \times P,J),\\
(\mathbf{q},\mathbf{p},x) \mapsto
(\mathbf{q},\mathbf{p},\phi^{\epsilon\, g(\mathbf{q})}(x)),
\end{gather*}
is $(i \oplus J_P, J)$--holomorphic. It follows that:

\begin{lem}
\label{lemkaehler}
The action
\[
(r^*,\id_P) \co (T^*S^m \times P,J) \to (T^*S^m \times P,J), \quad r \in
O(m) \subset O(m+1),
\]
is $J$--holomorphic and fixes $L$. Moreover, there are holomorphic
projections
\begin{gather*}
\widetilde{\pi_P} := \pi_P \circ \Phi^{-1} \colon (T^* S^m \times P,J) \to (P,J_P),\\
\pi_{T^*S^m}  \co (T^*S^m \times P,J) \to (T^*S^m,i),
\end{gather*}
where
\begin{gather*}\pi_{T^*S^m}=\pi_{T^*S^m} \circ \Phi^{-1}, \\
\widetilde{\pi_P} (L)=\Pi_L(\Lambda) \subset P,
\end{gather*}
and $\pi_P \co T^* S^m \times P \to P$ denotes the canonical projection.
\end{lem}

We let $\pi_{m+1} \co \C^{m+1} \to \C$ denote the holomorphic projection
onto the last complex coordinate; this projection induces a Lefschetz
fibration
\[
\pi_{m+1} \co A^m_1 \to \C
\]
whose critical points are $N$ and $S$ with the
critical
values $1$ and $-1$, respectively. We will write
\[
\pi := \pi_{m+1} \circ \Psi \co T^*S^m \to \C
\]
for the induced Lefschetz fibration of $T^*S^m$. The image of $(\pi\circ
r^* \circ
\pi_{T^*S^m} )(L) \subset \C$ for any $r \in O(m)$ is shown schematically
in \autoref{figfibration}.

Finally, the set $\pi_{T^*S^m}(L)$ is shown in \autoref{figprojection}
in the case $m=1$.

\begin{figure}[ht!]
\labellist\small
\pinlabel $x$ at 178 47 \pinlabel $iy$ at 85 107
\pinlabel $\pi(\pi_{T^*S^m}(L))$ at 35 77
\pinlabel $1$ at 147 37 \pinlabel $-1$ at 20 37
\pinlabel $e^{\epsilon/2}\epsilon/2\max_\Lambda h$ at 125 80
\pinlabel $e^{\epsilon/2}\epsilon/2\min_\Lambda h$ at 125 15
\endlabellist
\includegraphics[width=200px]{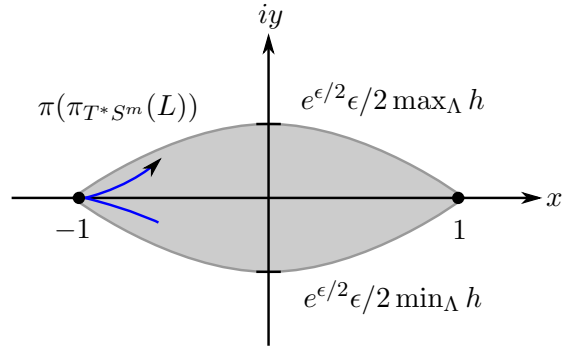}
\caption{The image of $L \subset T^*S^m \times P$ under the
holomorphic projection $\pi\circ \pi_{T^*S^m}$, where $\pi \co T^*S^m \to
\C$ is the Lefschetz fibration as described~above.
 Here
$\pi(N)=1$ and $\pi(S)=-1$. The arrow schematically
depicts the~behavior of the boundary near the positive puncture for a
$J$--holomorphic disk having boundary on $L$ and a positive puncture
contained inside
$\pi_{T^*S^m}^{-1}(S)$.}
\label{figfibration}
\end{figure}

\subsection{Transversality results}
\label{sectrans}

Recall
that a pseudo-holomorphic disk
\[
u \co (D^2,\partial D^2) \to (T^*S^m \times P,L)
\]
is said to have a \emph{positive} (resp.\ \emph{negative})
boundary puncture at $p \in \partial D^2$ if $u(p)$ is a
double point of $L$ and
the $z$--coordinate of $L$ jumps to a higher value
(resp.\
lower value) at the puncture when traversing the
boundary of the disk according to the orientation.

The following lemma is important and relates the transversality of a
$J$--holomorphic disk inside $\pi^{-1}(S) \subset T^*S^m \times P$ having
boundary on $L$ with the transversality of its projection to $P$, which
is a $J_P$--holomorphic disk having boundary on $\Pi_L(\Lambda)$. Observe
that the analogous result for the corresponding disks inside $\pi^{-1}(N)
\subset T^*S^m \times P$ is false in general, which is the reason why
computing the full differential $\partial_L$ in terms of $\partial_\Lambda$
is a more difficult problem.

\begin{lem}
\label{lemliftedtrans}
The $J$--holomorphic disk $u=(S,\widetilde{u})$ in $T^*S^m \times P$ having
boundary on~$L$ is transversely cut out if and only if the $J_P$--holomorphic
disk $\widetilde{u}$ having boundary on $\Pi_L(\Lambda)$ is transversely
cut out.
\end{lem}
\begin{proof}
We let $\widetilde{D}_{\widetilde{u}}$ and $D_u$ denote the linearizations
of the Cauchy--Riemann operators $\partial_{J_P}$ and $\partial_J$ at
the solutions $\widetilde{u}$ and $u$, respectively. We also linearize
the boundary condition and we use $\ker\widetilde{D}_{\widetilde{u}}$
and $\ker D_u$ to denote the spaces of solutions to the corresponding
boundary value problems. From the definition of the index of a Fredholm
operator it follows that

\begin{gather*}
\dim \ker \widetilde{D}_{\widetilde{u}} \ge \indx \widetilde{D}_{\widetilde{u}},\\
\dim \ker D_u \ge \indx D_u,
\end{gather*}
with equality if and only if the corresponding solution is transversely
cut out.

In this case, one checks that the Fredholm indices for the linearized
boundary value problems satisfy
\[
\indx \widetilde{D}_{\widetilde{u}}=\indx D_u.
\]
This can be seen by an explicit calculation utilizing \autoref{lembijection}
and the index formula for the pseudo-holomorphic disks under consideration;
see eg~\cite{NILSIRTNPO}.

There is a projection
\[
T\widetilde{\pi_P} \co \Gamma(u^*T(T^*S^m \times P)) \to \Gamma(u^*TP)
\]
of sections, induced by the differential of $\widetilde{\pi_P}$.

The ``only if'' part is straight-forward: the kernel of
$\widetilde{D}_{\widetilde{u}}$ lifts to the kernel of $D_u$
under~$T\widetilde{\pi_P}$. Hence, if $\widetilde{u}$ is not transversely
cut out, ie if $\dim \ker \widetilde{D}_{\widetilde{u}}>\indx
\widetilde{D}_{\widetilde{u}}$, it follows that $\dim \ker D_u>\indx
\widetilde{D}_{\widetilde{u}}=\indx D_u$ and $u$ is not transversely cut
out either.\looseness=-1

We proceed to show that $T\widetilde{\pi_P}$ restricts to an injection
\[
T\widetilde{\pi_P}|_{\ker D_u} \co \ker D_u \to \ker
\widetilde{D}_{\widetilde{u}}.
\]
Observe that the sought statement would follow from the
injectivity of $T\widetilde{\pi_P}|_{\ker D_u}$. Namely, $\ker
\widetilde{D}_{\widetilde{u}}=\indx \widetilde{D}_{\widetilde{u}}$
implies that

\[
\ker D_u \le \ker \widetilde{D}_{\widetilde{u}}=\indx
\widetilde{D}_{\widetilde{u}}=\indx D_u
\]
and hence that $\ker D_u=\indx D_u$.

The fact that $T\widetilde{\pi_P}|_{\ker D_u}$ is an injection
is shown as follows (it can be seen as an infinitesimal version
of \autoref{lemtrans}). Assume that we are given $\zeta \in \ker
T\widetilde{\pi_P}|_{\ker D_u}$. The fact that $\zeta$ is in the kernel
of $T\widetilde{\pi_P}$ implies that it can be considered as a holomorphic
map from the disk into
\[
(T_S(T^*S^m),i) \simeq (\C^m,i).
\]
We consider the $j^{\rm th}$ component $\zeta_j$ of this map. There is a constant
$c>0$ for which the following holds. Recall that $h \co \Lambda \to \R$
is the $z$--coordinate of $\Lambda$. Let $p_+$ and $p_-$,
with $h(p_+)>h(p_-)$,
be the two lifts of the double point $p \in \Pi_L(\Lambda)$ which is the
image of the positive puncture of $\widetilde{u}$. Further, let
$$
h_{\mathrm{max}} := c \max_\Lambda h,
\quad
h_{\mathrm{min}} := c \min_\Lambda h,
\quad
h_+ := c \,h(p_+),
\quad
h_- := c \, h(p_-).
$$
For a suitable choice of holomorphic coordinates and constant $c>0$, the
linearized boundary condition implies that $\zeta_j$ is contained inside
the double cone $C_+ \cup C_-$, where
$$
C_+ := \{ x+iy\mid  y \in [h_{\mathrm{min}}x,h_{\mathrm{max}}x],\, x \ge 0 \}
\subset \C
\:\:\text{and}\:\:
C_- := -C_+,
$$
as shown in \autoref{figtangentcone}. Furthermore, $\zeta_m$ necessarily
vanishes at the positive puncture, while it is asymptotic to the line
\begin{itemize}
\item $y=h_- x$ when approaching the positive puncture along the boundary
in the direction of the orientation, and
\item $y=h_+x$ when approaching the positive puncture along the boundary
against the direction of the orientation.
\end{itemize}

The open mapping theorem implies that the interior of $\zeta_m$ is mapped
to either $C_+$ or~$C_-$. Together with the asymptotics of $\zeta_m$ at
the positive puncture, the open mapping theorem now shows that $\zeta_m$
in fact must vanish identically. Hence $\zeta \equiv 0$, which establishes
the injectivity of $T\widetilde{\pi_P}|_{\ker D_u}$.
\end{proof}

\begin{figure}[ht!]
\labellist\small
\pinlabel $x$ at 149 51
\pinlabel $h_+$ [l] at 76 76
\pinlabel $h_{\mathrm{max}}$ [l] at 76 88
\pinlabel $h_-$ [l] at 76 36
\pinlabel $h_{\mathrm{min}}$ [l] at 76 27
\pinlabel $iy$ at 73 110
\pinlabel $1$ at 121 42
\pinlabel $C_+$ at 128 65
\pinlabel $C_-$ at 20 40
\endlabellist
\includegraphics[width=200px]{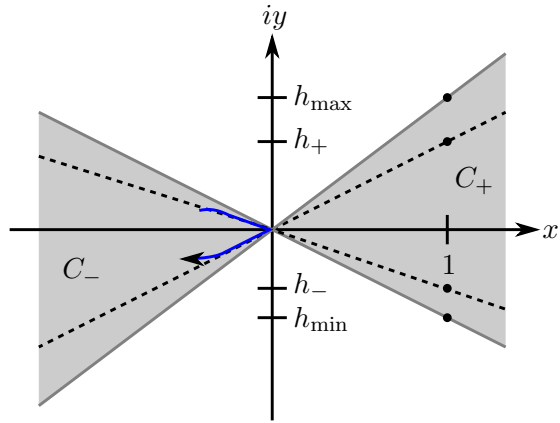}
\caption{The double cone $C_- \cup C_+$ depicts the
projection of the linearized boundary condition near a puncture in
$\pi_{T^*S^m}^{-1}(S)$. The arrow shows the behavior of a solution to the
linearized problem along the boundary near the positive puncture, given
that it is non-vanishing.}
\label{figtangentcone}
\end{figure}

We continue to show that the above lemma
can indeed be applied to
the $J$--holomorphic disks contributing to $\partial_L(s)$ for $s \in
\mathcal{Q}_S$.

\begin{lem}
\label{lemtrans} Let $u$ be a $J$--holomorphic disk in $T^*S^m \times P$
having boundary on $L$ and exactly one positive boundary puncture. If the
positive puncture of $u$ is contained
in~$\pi_{T^*S^m}^{-1}(S)$, then all of $u$ is contained inside
$\pi_{T^*S^m}^{-1}(S)$. In
particular, such disks are in bijective correspondence with the
corresponding $J_P$--holomorphic disks in $P$ with boundary on
$\Lambda$. Finally, given that $J_P$ is regular, it follows that
these $J$--holomorphic disks are transversely cut out as well.
\end{lem}
\begin{proof}
In the case $m=1$, this follows eg~from \cite[Lemma 4.14]{OILCHAELI}. We
here provide a proof of the general case. The idea is to study the
image of $\pi_{T^*S^m} \circ u$ under the Lefschetz fibration $\pi
\co T^*S^m \to \C$. The image of $L$ under this projection is shown in
\autoref{figfibration}.

To that end, we first observe that $\pi \circ \pi_{T^*S^m} \circ u$ is
holomorphic. The
asymptotics at the positive puncture (see the arrow in
\autoref{figfibration} for the projection of the boundary in the case when
the projection is non-zero) together with the open mapping theorem implies
that $\pi \circ \pi_{T^*S^m} \circ u$ is constantly equal
to $-1 \in \C$ in some neighborhood of this puncture, and hence that
$\pi \circ \pi_{T^*S^m} \circ u \equiv -1$.

Under the biholomorphism $\Psi$, the fiber $\pi^{-1}(-1) \subset T^*S^m$
gets identified with the singular quadric
\[
Q:=\{ z_1^2 +\cdots+z_m^2=0\} \cap A_1^m \subset \C^{m+1},
\]
such
 that $S$ is identified with the unique singular point $(0,\ldots,0,-1)
\in Q$. Since $\pi_{T^*S^m} \circ u$ maps into this fiber and
the
boundary is mapped to
\[
\pi_{T^*S^m}(L) \cap \pi^{-1}(-1) =\{S\}
\]
(ie the singular point), the maximum principle now implies that
$\pi_{T^*S^m} \circ u \equiv S$. In other words, $u$ is contained
in $\pi_{T^*S^m}^{-1}(S)$.

Conversely, for each $J_P$ holomorphic disk $\widetilde{u}$ in $P$
having boundary on $\Pi_L(\Lambda)$ one immediately constructs a
$J$--holomorphic disk $u=(S,\widetilde{u})$ in $T^*S^m \times P$ having
boundary on $L$.

The transversality statement follows from \autoref{lemliftedtrans} above.
\vadjust{\goodbreak}
\end{proof}

\begin{lem}
\label{lemindex} A non-constant $J$--holomorphic disk $u$ in $T^*S^m \times P$
with boundary on~$L$ and exactly one positive puncture contained in
$\pi_{T^*S^m}^{-1}(N)$ has index
\begin{equation}
\label{eqindexeq}
\mathrm{index}(u)=\mathrm{index}(\widetilde{\pi_P} \circ u)+m- N_um,
\end{equation}
where $N_u$ is the number of negative punctures of $u$ contained in
$\pi_{T^*S^m}^{-1}(N)$.

If $\widetilde{\pi_P} \circ u$ is non-constant, under the additional
assumption that $J_P$ is a regular almost complex structure it follows that
\begin{equation}\label{eqindexineq}
\mathrm{index}(u) \ge m-N_um.
\end{equation}
If $\widetilde{\pi_P} \circ u \equiv p \in P$ is constant, it follows that
$u$ is a pseudo-holomorphic
strip whose negative puncture is contained in $\pi_{T^*S^m}^{-1}(S)$,
both
of whose punctures correspond to~$p \in \mathcal{Q}(\Lambda)$ and whose
index satisfies
\begin{equation}
\label{eqindexineqtw}
\mathrm{index}(u)=m-1 \ge 0.
\end{equation}
In the case $m=1$ this strip is, moreover, transversely cut out.
\end{lem}
\begin{proof}
Formula \eqref{eqindexeq} follows from a simple index calculation utilizing
\autoref{lembijection} and the index formula in \cite{NILSIRTNPO}.

Suppose that $\widetilde{\pi_P} \circ u$ is not constant. The inequality
\eqref{eqindexineq} follows from \eqref{eqindexeq} together with
the inequality $\mathrm{index}(\widetilde{\pi_P} \circ u) \ge 0$. To see
the latter inequality, recall that the projection
\[
\widetilde{\pi_P}  \co (T^*S^m \times P,L) \to (P,\Pi_L(\Lambda))
\]
is $(J,J_P)$--holomorphic by \autoref{lemkaehler}, and that $J_P$ is regular
by assumption, and hence that the index of the solution is equal to the
dimension of the moduli space
in which it is contained.

Suppose that $\widetilde{\pi_P} \circ u$ is constant. First,  it follows that
we have $u=(\widetilde{u},p)$, where $p \in \Pi_L(\Lambda)$ is a double point
and $\widetilde{u} \co D^2 \to T^*S^m$ is holomorphic. It follows that $u$
must be a strip whose positive puncture
(resp.\ negative puncture) is
the puncture in~$\mathcal{Q}_N$ (resp.\ $\mathcal{Q}_S$) corresponding
to $p$. From this fact we get \eqref{eqindexineqtw}.

Finally, the transversality claim in the case $m=1$ follows by an
explicit calculation
that is standard. To that end, observe that the
strip $\widetilde{u}$ is an embedding of a shaded strip, as shown in
\autoref{figprojection}.
\end{proof}

Recall that a \emph{broken $J$--holomorphic disk} consists of a connected
directed tree satisfying the following conditions. Let $V$ and $E$ denote
the vertex and edge sets of the tree, respectively. First, each vertex $v
\in V$ is assigned a $J$--holomorphic disk $u_v$ with exactly one positive
puncture. Second, each edge $e \in E$, also called a \emph{node} of the
broken configuration, is assigned a negative puncture $q_e$ of $u_v$ and a
positive puncture $p_e$ of $u_w$, where $v$ and $w$ are the starting point
and endpoint of $e$, respectively, and for which $u_v(q_e)=u_w(p_e)$ is
required to hold. By a positive (respectively, negative) puncture of a
broken $J$--holomorphic disk we mean a positive (respectively, negative)
puncture of one of the involved disks which does not correspond to any node.

\begin{lem}
\label{lemunbroken} Assume that $J_P$ is regular. A broken $J$--holomorphic
disk $u$ in $T^*S^m \times P$ with boundary on $L$, exactly one positive
puncture contained in $\pi_{T^*S^m}^{-1}(N)$, and all of its negative
punctures contained in $\pi_{T^*S^m}^{-1}(S)$, must be of positive index.
\end{lem}

\begin{proof}
Suppose that the broken disk consists of the $J$--holomorphic discs
$\{u_i\}_{i=1}^{\nu+1}$ and $\{v_i\}_{i=1}^\mu$, where $u_i$ has its positive puncture
contained in $\pi_{T^*S^m}^{-1}(N)$ and $v_i$ has its positive puncture
contained in $\pi_{T^*S^m}^{-1}(S)$. Furthermore, we may order the disks
$\{u_i\}$ so that $\widetilde{\pi_P}\circ u_i$ is constant if and only
if $k < i \le \nu+1$ for some $0\le k \le \nu+1$. Observe that $\mu+\nu$
is the total number of nodes of the broken disk.

The index of the broken disk is computed to be
\[
I:=\nu+\sum_{i=1}^{\nu+1} \mathrm{index}(u_i)+\mu+\sum_{i=1}^\mu
\mathrm{index}(v_i),
\]
where \autoref{lemtrans} implies that $\mu + \sum_{i=1}^\mu
\mathrm{index}(v_i) \ge 0$.

In the case $k=0$ \autoref{lemindex} implies that $\nu=0$ and, since
$\mu+\nu>0$ by assumption, it follows that $\mu>0$. In this case,
\eqref{eqindexineqtw} implies that
\[
I \ge \mu + \mathrm{index}(u_0)=\mu + (m-1)>0.
\]

In the case $k>0$, \eqref{eqindexineq} and \eqref{eqindexineqtw}
show that
\begin{eqnarray*}
\lefteqn{I \ge \nu+\sum_{i=1}^{\nu+1} \mathrm{index}(u_i)}  \\
& \ge & \nu +\sum_{i=1}^k(m-mN_{u_i})+(\nu+1-k)(m-1) \\
& = & \nu +m\left(k-\sum_{i=1}^k N_{u_i}\right)+(\nu+1-k)(m-1).
\end{eqnarray*}
Finally, the assumption that all negative punctures of the broken disk are
contained in $\pi_{T^*S^m}^{-1}(S)$ together with the fact that the disks
$v_i$ have all negative punctures contained in $\pi_{T^*S^m}^{-1}(S)$ by
\autoref{lemtrans}
gives the inequality $\nu \ge \sum_{i=1}^kN_{u_i}$. In
conclusion, we have
\[\nu+\sum_{i=1}^{\nu+1} \mathrm{index}(u_i) \ge \nu+m(k-\nu)+(\nu+1-k)(m-1)=m-1+k>0.\]
\end{proof}

\subsection{A sign computation in the case $m =1$}

\begin{lem}
\label{lembdy} Let $u$ be a $J$--holomorphic disk in $T^* S^1 \times P$ having
boundary on $L$. If the projection $\pi_{T^*S^1} \circ u$ is non-constant,
then the image of $\pi_{T^*S^1} \circ u$ intersects the interior of exactly
one of the two subsets
\begin{gather*}
T^*\{ (\cos \theta,\sin \theta); \: \pi/2 \le \theta \le 3\pi/2 \} \subset T^*S^1, \\
T^*\{ (\cos \theta,\sin \theta); \: -\pi/2 \le \theta \le \pi/2 \} \subset T^*S^1.
\end{gather*}
\end{lem}
\begin{proof}
By \autoref{lemtrans}, we may suppose that the positive puncture
of the disk is contained in $\pi_{T^*S^1}^{-1}(N)$, since it
would otherwise have a constant projection to $S \in\nobreak T^*S^1$.%

Further, unless its projection to $T^*S^1$ is constant, an interior point
of the disk cannot be mapped into $\pi_{T^*S^1}^{-1}\{ S
\cup N\}$, as follows from the open mapping theorem applied to the holomorphic
projection $\pi_{T^*S^1}\circ u$. To that end, we refer to
\autoref{figprojection} for a
description of the image of the boundary condition $\pi_{T^*S^1}(L)$ under
this projection.
\end{proof}

\begin{figure}[ht!]
\labellist\small
\pinlabel $\pi_{T^*S^1}(L)$ at 192 90
\pinlabel $S^1$ at 165 120
\pinlabel $T^*S^1$ at 173 25
\pinlabel $N$ at 73 160
\pinlabel $S$ at 73 80
\pinlabel $r$ at 317 129
\pinlabel $S^1$ at 380 70
\pinlabel $N$ at 318 200
\pinlabel $S$ at 314 40
\endlabellist
\includegraphics[width=260px]{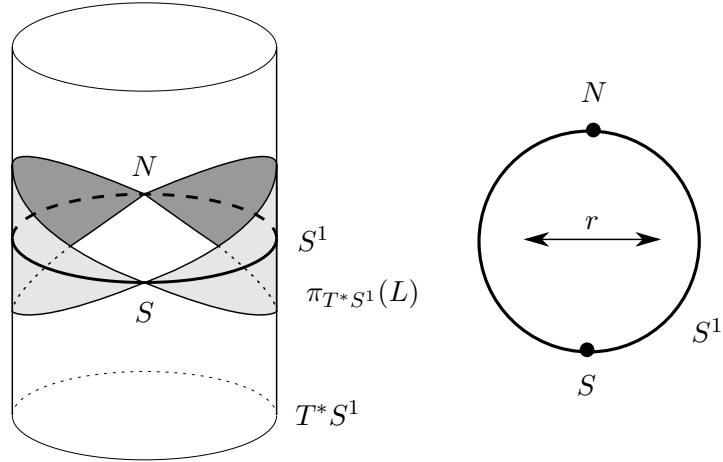}
\label{figprojection}
\caption{The image of $L \subset P \times T^*S^1$ under the
projection $\pi_{T^*S^1}$. Observe that $\pi_{T^*S^1}(L)$ is invariant
under $r^*$, where $r \in O(1)$.}
\vskip-10.5pt
\end{figure}

\begin{lem}
\label{lemsign}
The $J$--holomorphic involution $(r^*,\mathrm{id}_P)$ acts non-trivially
on the $J$--holomorphic disks in $T^*S^1 \times P$ having boundary on $L$
and non-constant projection to~$T^*S^1$.

In the case when $\Lambda$ is spin and the $J$--holomorphic disk
$u=(\widetilde{u},p)$, $p \in P$, is a non-trivial strip in $T^*S^1 \times
P$ with boundary on $L$, the involution applied to $u$ moreover reverses
the coherent orientation.
\end{lem}
\begin{proof}
The first claim follows immediately from \autoref{lembdy}.

To check the effect on the coherent orientation of the moduli space containing
a $J$--holomorphic strip $u$ as in the claim, we argue as follows. By
\autoref{lembdy} it follows that $u$ and $r^* \circ u$ have
homotopic boundaries after ``capping off'' the boundary with a choice
of capping path at each puncture. The claim follows since the involution
$(r^*,\mathrm{id}_P)$ restricts to an orientation-reversing
involution on $L$ and
since it does not fix~$u$.
\end{proof}

\subsection{Proof of \texorpdfstring{\autoref{maininclinviig}}{Theorem 3.1}}
\label{proof}

We choose the perturbation $L$ of $S^m \times \Lambda$ as
constructed in \autoref{secpert}, for which we have the decomposition
$\mathcal{Q}(L)=\mathcal{Q}_N \sqcup \mathcal{Q}_S$ of the Reeb chords. The
grading-preserving
bijection between the Reeb chords $\mathcal{Q}_S$ on $L$ and the
Reeb chords on $\Lambda$, as established in \autoref{lembijection}, thus
induces a graded unital algebra inclusion
\[
\iota \co \mathcal{A}_\bullet(\Lambda) \hookrightarrow
\mathcal{A}_\bullet(L).
\]
Perturb the almost complex structure $J$ constructed in \autoref{secacs}
to a regular, almost complex structure $J'$. For a
sufficiently small such perturbation, \autoref{lemtrans} implies
that $\iota$ is a chain map. Here we have used the fact that the
number of solutions in a transversely cut-out rigid moduli space
stays the same after a sufficiently small such perturbation.

Observe that $\iota$ has a left inverse
\[
\pi \co \mathcal{A}_\bullet(L) \to \mathcal{A}_\bullet(L) /\langle
\mathcal{Q}_N\rangle =\mathcal{A}_\bullet(\Lambda)
\]
on the algebra level, which is induced by the quotient with the two-sided
ideal generated by $\mathcal{Q}_N$. We will show that, given that $J'$
is sufficiently close to $J$, $\pi$ is in fact a chain map, from which the
theorem follows.

The fact that $\pi$ is a chain map is equivalent to $\partial_L$ preserving
the two-sided ideal $\langle \mathcal{Q}_N\rangle$. To that end, for each
$c \in \mathcal{Q}_N$, we need to show that $\partial_L(c)$ may be assumed
to consist of a sum of words all which contain at least one letter from
$\mathcal{Q}_N$.

Taking the limit $J' \to J$, the solutions of $J'$--holomorphic disks
converge to (possibly broken) $J$--holomorphic disks by the Gromov--Floer-type
compactness used in \cite{LCHIPXR}. In particular, if the set of (possibly
broken) solutions of $J$--holomorphic disks is empty for some specifications
of the punctures, then the same is true for the corresponding moduli space of
$J'$--holomorphic disks given that $J'$ is a sufficiently small perturbation
of $J$. \autoref{lemunbroken} implies that the set of broken $J$--holomorphic
disks of index zero which (after perturbing
$J$) could contribute to a term in the expression $\partial_L(c)$ all
have at least one negative puncture in $\mathcal{Q}_N$. In other words,
the $J'$--holomorphic disks that contribute to words in the expression
$\partial_L(c)$ containing no letters from $\mathcal{Q}_N$ correspond to
\emph{unbroken} $J$--holomorphic disks.

\textbf{Case $\mathbf{m} \boldsymbol{\ge} \mathbf{2}$}\quad\autoref{lemindex} implies
that there are no such unbroken disks of index zero, which shows the claim.

\textbf{Case $\mathbf{m}\boldsymbol{=}\mathbf{1}$}\quad
 \autoref{lemindex} implies that such disks of index zero
are necessarily $J$--holomorphic strips
that are constant when projected
to $P$. By \autoref{lemsign}, there are exactly two such strips, which,
moreover, are transversely cut-out solutions equipped with opposite coherent
orientations. In other words, they do not contribute to $\partial(c)$.

\section{Examples beyond augmentable Legendrians}\label{exfindimreprbnaug}
Recall that for Legendrian submanifolds whose Chekanov--Eliashberg
algebra admits an augmentation, \eqref{weekarniineqbcfgagbn}
was proven by Ekholm, Etnyre and Sullivan \cite{OILCHAELI}.
Ekholm, Etnyre and Sabloff \cite{ADESFLCH}
  later improved it to
\eqref{arnineqmrepbnbad} under the same assumptions.
In addition, given that $\Lambda$ is spin and
has a Chekanov--Eliashberg algebra admitting an augmentation into
$R=\Z$, $\Q$, $\R$
or $\Z_m$, we may take $b_i:=\mathrm{rk}\, H_i (\Lambda;R)$
in the inequality.

There are certain examples of Legendrian knots in $\C\times\R$ due to
Sivek~\cite{TCHOLKWMTBI}, which will be discussed below, whose
characteristic algebras admit $2$--dimensional representations over $\Z_2$,
but no $1$--dimensional representations over $\Z_2$. We will use these
examples to construct
Legendrian submanifolds $\Lambda \subset \C^n\times \R$ having
Chekanov--Eliashberg algebras that do not admit augmentations in any unital
commutative ring $R$, but for which there exists a finite-dimensional
representation $\rho: \MC_{\Lambda}\to \mathrm{M}_{k}(\Z_2)$ for
some $k>1$.

Recall that the \emph{Kauffman bound}
\[
\operatorname{tb}(\Lambda) \le \operatorname{min-deg}_{a} F_{\Lambda}(a,x)
- 1
\]
is satisfied for any Legendrian knot $\Lambda \subset \C \times \R$,
as shown by Rudolf \cite{ACBLP}, where $F_\Lambda(a,x)$ is the so-called
Kauffman polynomial for the underlying smooth knot of $\Lambda$. We will
need the following strong condition for the existence of an augmentation.
\begin{prop}[Rutherford \cite{TBNKPARIOALL}; Henry and Rutherford \cite{RPAAOFF}]
\label{propkauffbound}
Let $\F$ be an arbitrary field and $\Lambda \subset \C \times \R$ a Legendrian
knot. The Kauffman bound
\[
\operatorname{tb}(\Lambda) \le \operatorname{min-deg}_{a} F_{\Lambda}(a,x)
- 1
\]
is an equality if and only if the Chekanov--Eliashberg algebra of $\Lambda$
defined with Novikov coefficients $\F[H_1(\Lambda)]$ admits an ungraded
augmentation into $\F$.
\end{prop}
\begin{proof}
The proof is similar to the proof of  \cite[Proposition 3.2]{TCHOLKWMTBI},
which considers the case $\F=\Z_2$. The Legendrian knot admits an ungraded
ruling if and only if the Kauffman bound is an equality, by
\cite[Theorem~3.1]{TBNKPARIOALL}, and the existence of an ungraded ruling is equivalent
to the existence of an ungraded augmentation in $\F$ by
\cite[Theorem~3.4]{RPAAOFF} (given that the Chekanov--Eliashberg algebra is defined with
Novikov coefficients).
See \cite{AAROLK} for an alternative proof of
the latter fact.
\end{proof}

\begin{rem}
Regarding higher-dimensional representations of the characteristic algebra,
the following can be said.
If the characteristic algebra $\MC_{\Lambda}$ of a Legendrian knot
$\Lambda\subset \C\times\R$ admits a finite-dimensional
representation over $\Z_2$, then $\Lambda$ maximizes the Thurston--Bennequin
invariant within its
topological knot type \cite[Theorem 1.2]{SOLKAROTCEA}. Furthermore, the existence of such a
representation depends only on $\operatorname{tb}(\Lambda)$ and the
topological
type of $\Lambda$.
\end{rem}

We first prove the following lemma, which shows how
finite-dimensional representations of characteristic algebras behave
under the operation of connected sum, as defined  by Etnyre
and Honda \cite{OCSALK}.

\begin{lem}\label{rzeconsumlegknofun}
Let $\Lambda_1$ an $\Lambda_2$ be two Legendrian knots in $\C\times\R$.
\begin{enumerate}
\item
If the characteristic algebra $\MC_{\Lambda_i}$ admits a
$k_i$--dimensional (graded) representation over a field $\F$ for
$k_i \ge 1$ and $i=1$, $2$, then  $\MC_{\Lambda_1 \# \Lambda_2}$ admits a
$k_{1}k_{2}$--dimensional
(graded) representation over $\F$.
\item
If one of $\Lambda_i$, $i=1$, $2$, satisfies the strict inequality
$$
\operatorname{tb}(\Lambda_i) < \operatorname{min-deg}_{a} F_{\Lambda_i}(a,x)
- 1
$$
for the Kauffman bound, then the Chekanov--Eliashberg algebra of
$\Lambda_1\# \Lambda_2$ does not admit an augmentation into any unital
commutative ring $R$.
\end{enumerate}
\end{lem}

\begin{proof}
We first prove $(1)$.
Observe that there is an exact Lagrangian cobordism~$L$ from
$\Lambda_1\sqcup \Lambda_2$ to $\Lambda_{1}\# \Lambda_{2}$ by
eg~\cite{LASALCN} or \cite{LKAELC}. It follows that there exists a unital
algebra morphism $f_{L}\co
\MC_{\Lambda_{1}\# \Lambda_{2}}\to \MC_{\Lambda_{1}}\otimes
\MC_{\Lambda_{2}}$. If $\rho_{i}\co \MC_{\Lambda_{i}} \to
\mathrm{M}_{k_i}(\F)$ is a (graded) $k_i$--dimensional representation
of $\MC_{\Lambda_{i}}$, then the pullback of the tensor product of
(graded) representations $\rho_1$ and $\rho_2$
\begin{align*}
(f_L)^{\ast}(\rho_1\otimes \rho_2)\co \MC_{\Lambda_{1}\#  \Lambda_{2}}
\to \mathrm{M}_{k_{1}k_{2}}(\F)
\end{align*}
is a (graded) $k_{1}k_{2}$--dimensional representation of
$\MC_{\Lambda_{1}\# \Lambda_2}$.

We then prove $(2)$. First, observe that for any unital commutative ring $R$
there is a maximal ideal $m \subset R$. Composing the augmentation to $R$
with the quotient projection $R\to R/m$ thus induces an augmentation into the
field $R/m$. It thus suffices to show the statement when $R=\F$ is a field.

Without loss of
generality, we assume that the Kauffman bound for $\Lambda_1$ is not an
equality, ie that
\begin{align*}
\operatorname{tb}(\Lambda_1) < \operatorname{min-deg}_{a} F_{\Lambda_1}(a,x)
- 1.
\end{align*}
In addition,
$\operatorname{tb}(\Lambda_1\# \Lambda_2)=\operatorname{tb}(\Lambda_1)+\operatorname{tb}(\Lambda_2)+1$ \cite[Lemma 3.3]{OCSALK}.
Hence, using the Kauffman bounds for $\operatorname{tb}(\Lambda_1)$ and
$\operatorname{tb}(\Lambda_2)$, we get that
\begin{align*}
\operatorname{tb}(\Lambda_1\#\Lambda_2)=\operatorname{tb}(\Lambda_1)+\operatorname{tb}(\Lambda_2)+1<\operatorname{min-deg}_{a}
F_{\Lambda_1}(a,x)+\operatorname{min-deg}_{a} F_{\Lambda_2}(a,x)-1.
\end{align*}
Further, as shown in \cite{OKPOL}, we have the equality
\begin{align*}
\operatorname{min-deg}_{a} F_{\Lambda_1\#
\Lambda_2}(a,x)=\operatorname{min-deg}_{a}
F_{\Lambda_1}(a,x)+ \operatorname{min-deg}_{a} F_{\Lambda_2}(a,x),
\end{align*}
and hence we see that the Kauffman bound
\begin{align*}
\operatorname{tb}(\Lambda_1\#  \Lambda_2) < \operatorname{min-deg}_{a}
F_{\Lambda_1\#
\Lambda_2}(a,x)  - 1
\end{align*}
for $\Lambda_1\#  \Lambda_2$ is not an equality either. Consequently,
\autoref{propkauffbound} shows that the Chekanov--Eliashberg algebra of
$\Lambda_1\#  \Lambda_2$ defined using Novikov coefficients does not admit an
augmentation into $\F$. In particular, it follows that its Chekanov--Eliashberg
algebra defined without Novikov coefficients does not admit an augmentation
either.
\end{proof}

\begin{figure}[t]
\center
\includegraphics[width=5cm]{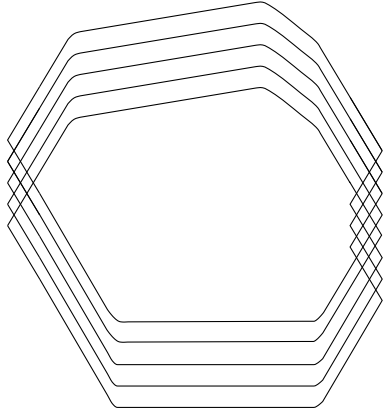}
\caption{The front projection of $\Lambda_{5,-8}$}
\label{stsrxfodrepbig}
\end{figure}

\begin{example}\label{firstfoundstfdrfro}
Sivek
 \cite{TCHOLKWMTBI} considered a Legendrian knot
$\Lambda_{p,-q}$, which is a Legendrian representative of
$(p,-q)$--torus knot with $p\geq 3$ odd and $q>p$. The front
projection of $\Lambda_{5,-8}$ is shown in
\autoref{stsrxfodrepbig}. Sivek proved that the characteristic
algebra $\MC_{\Lambda_{p,-q}}$ admits a representation $\rho\co
\MC_{\Lambda_{p,-q}}\to \mathrm{M}_{k}(\Z_2)$ for $k=2$ but not for
$k=1$; see~\cite{TCHOLKWMTBI}. In addition, observe that the proof
that $\MC_{\Lambda_{p,-q}}$ does not admit a $1$--dimensional
representation over $\Z_2$ is based on the fact that the Kauffman bound of
$\Lambda_{p,-q}$ is a strict inequality
$-pq = \operatorname{tb}(\Lambda_{p,-q})< -pq + q - p$; see
\cite{OTIOLMTL,TCHOLKWMTBI}.
Hence, using the same argument as in the proof of
\autoref{rzeconsumlegknofun}, one gets that $\MC_{\Lambda_{p,-q}}$ does
not admit an augmentation into a unital commutative ring. Observe that
$r(\Lambda_{p,-q})=q-p$ (if the orientation is such
that the cusps in the left part of \autoref{stsrxfodrepbig} are
downward cusps). Since we will use it in the proof of
\autoref{applsptwtherffdrbnaugalk} below, observe that the Maslov
class of $\Lambda_{p,-q}$ is given by $\mu(\Lambda_{p,-q})=2(q-p)$.
\end{example}

We
now discuss another family of examples.

\begin{figure}[b]
\labellist\small
 \pinlabel $a_1$ at 330 345 \pinlabel $a_2$ at 330 310
\pinlabel $a_3$ at 330 275 \pinlabel $a_{2k}$ at 330 60 \pinlabel
$a_{2k+1}$ at 343 25
\endlabellist
\includegraphics[width=6cm]{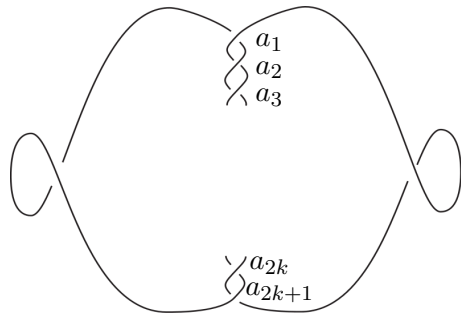}
\caption{The Lagrangian projection of $T_{2,k}$} \label{Tanoldbfny}
\vskip-2.5pt
\end{figure}

\begin{example}\label{ecexampleconsmnbnto}
Observe that from \autoref{firstfoundstfdrfro} and
\autoref{rzeconsumlegknofun} we can get many other examples of
Legendrian knots $\Lambda$ whose characteristic algebras admit
finite-dimensional representations over $\Z_2$, but whose
Chekanov--Eliashberg algebras do not admit augmentations in any
unital commutative ring $R$. Say, consider $\Lambda_{p,-q}\#
T_{2,k}$, where $T_{2,k}$ is a Legendrian representative of torus
$(2,k)$--knot -- see \autoref{Tanoldbfny} -- where $p\geq 3$
is odd and
fixed, $k\in \N$ is fixed and $q>p$. Since
the Chekanov--Eliashberg
algebra of $T_{2,k}$ admits an augmentation into $\Z_2$-- see
\cite{LKAELC}-- its characteristic algebra has a $1$--dimensional
representation over $\Z_2$. Then we apply
\autoref{rzeconsumlegknofun} and see that $\Lambda_{p,-q}\#
T_{2,k}$ satisfies the sought properties. In addition, since
$r(T_{2,k})=0$, observe that $r(\Lambda_{p,-q}\#  T_{2,k})=q-p$ and
one can find infinitely many such knots which are pairwise not
Legendrian isotopic. We also note that the Maslov class
$\mu(\Lambda_{p,-q}\#  T_{2,k})$ equals $2(q-p)$.
\end{example}

\subsection{Proof of \texorpdfstring{\autoref{applsptwtherffdrbnaugalk}}{Theorem 1.9}}
Consider a family of Legendrian knots $(\widetilde{\Lambda}_l)_{l\in
\N}$ satisfying the following three properties:
\begin{enumerate}
\item $\mu(\widetilde{\Lambda}_{l_1})\neq \mu(\widetilde{\Lambda}_{l_2})$
whenever $l_1\neq l_2$ in $\N$.
\item $\A(\widetilde{\Lambda}_l)$ does not admit an augmentation to
any unital commutative ring $R$.
\item $\MC_{\widetilde{\Lambda}_l}$ admits a finite-dimensional
representation $\rho_{l}\co \MC_{\widetilde{\Lambda}_l}\to \mathrm M_{k}(\Z_2)$,
$k>1$.
\end{enumerate}

We can use a family $\widetilde{\Lambda}_{l}:=\Lambda_{p,-p-l}$,
 described in \autoref{firstfoundstfdrfro},
 with fixed $p\geq 3$
odd and $l\in \N$, or a family
$\widetilde{\Lambda}_{l}:=\Lambda_{p,-p-l}\#  T_{2,k}$,
described in \autoref{ecexampleconsmnbnto},
 where $l\in \N$
and
we have fixed $p\geq 3$ odd and $k\in \N$. We also
observe that any family of Legendrian knots $(\Lambda'_k)_{k\in \N}$
such that $\A(\Lambda'_k)$ admits an augmentation in $\Z_2$ and
$\mu(\Lambda'_{k_1})=\mu(\Lambda'_{k_2})$ for all $k_1$, $k_2\in \N$
will lead to a family $(\Lambda_l)_{l\in \N}$, where
$\widetilde{\Lambda}_{l}:=\Lambda_{p,-p-l}\#  \Lambda'_{k}$, which satisfies
the sought properties.

We now define
$\Lambda_{l}:=\Sigma_{S^{m_s}}\cdots\Sigma_{S^{m_1}}\widetilde{\Lambda}_{l}$.
Observe that, from property (1) and the result of Lambert-Cole
\cite{LP}, it follows that $\mu(\Lambda_{l_1})\neq
\mu(\Lambda_{l_2})$ for all $l_1\neq l_2$ in $\N$. Therefore, the
$\Lambda_{l}$ are pairwise not Legendrian isotopic. From the first
two parts of \autoref{cormrftspgdsy} and properties (2) and (3) it
follows that $\MC_{\Lambda_l}$ admits a finite-dimensional
representation $\rho_{l}\co \MC_{\Lambda_l}\to \mathrm M_{k}(\Z_2)$, $k>1$,
but that $\A(\Lambda_l)$ does not admit an augmentation to any unital
commutative ring $R$. Finally, we see that
\autoref{mainthartinchlocbig} is applicable to $\Lambda_{l}$.
This finishes the proof.

\section{Limitations of the argument}\label{arglimsecttomonbg}
Note that our proof of \autoref{mainthartinchlocbig} is based on
the fact that if $\MC_{\Lambda}$ admits a $k$--dimensional
representation $\rho$, then we can reduce the
$\MC_{\Lambda}$--$\MC_{\Lambda'}$--bimodule $\MC(\Lambda,\Lambda')$ to
the $\mathrm{M}_k(\F)$--$\mathrm{M}_k(\F)$--bimodule
$\MC_{\rho}(\Lambda,\Lambda')$, which leads to rank inequality
\eqref{mineqanofthernulgtr} (which in the simplest situation of a
short exact sequence of finitely generated abelian groups is
equivalent to an application of the rank--nullity theorem). In order
to have this type of inequality for ranks of bimodules, the minimal
imaginable requirement on $\MC_{\Lambda}\otimes \MC_{\Lambda'}^{\mathrm{op}}$
is to have
a well-defined notion of rank for free left
$\MC_{\Lambda}\otimes \MC_{\Lambda'}^{\mathrm{op}}$--modules (the so-called
IBN property) which, moreover, satisfies the property that the rank
of a free module cannot be exceeded by the rank of a free submodule
(the so-called strong rank property).

\begin{rem}\label{rkproptostrrankpropnatbnvintrghsk}

Assume that we are given a Legendrian submanifold $\Lambda\subset P\times
\R$ for which $\MC_{\Lambda}$ does not satisfy the rank property.
From
the fact that there exists a natural unital homomorphism
$\MC_{\Lambda}\to \MC_{\Lambda}\otimes \MC_{\Lambda'}^{\mathrm{op}}$ and
\autoref{rhomoibnrcondpres} it follows that $\MC_{\Lambda}\otimes
\MC_{\Lambda'}^{\mathrm{op}}$ does not satisfy the rank property
either. Hence, using \autoref{genpictimplibnrpstrrkpr} it follows that
$\MC_{\Lambda}\otimes
\MC_{\Lambda'}^{\mathrm{op}}$ in particular does not satisfy the strong
rank property.
\end{rem}

We now proceed to prove the following simple fact:

\begin{fact}\label{factnonrprop}
If R is a unital ring for which there exist elements $x$, $y$, $p$, $q\in R$
satisfying
$$
p(1-yx)q=1\quad  \mbox{and}\quad xy=1,
$$
then $R$ does not satisfy the rank property.
\end{fact}

\begin{proof}\
Consider a column vector $(x, p(1-yx))^{T}$ and a row vector
$(y,(1-yx)q)$. It is easy to see that
\[\let\bmatrix\pmatrix
\let\endbmatrix\endpmatrix
\begin{bmatrix}
    x\\
    p(1-yx) \\
\end{bmatrix} \cdot
\begin{bmatrix}
    y & (1-yx)q \\
\end{bmatrix}
=
\begin{bmatrix}
    xy & (x - xyx)q  \\
    p(y - yxy) & p(1-2yx + yxyx)q \\
\end{bmatrix} =
\begin{bmatrix}
    1 & 0  \\
    0 & 1 \\
\end{bmatrix},
\]
and hence using \autoref{nonibnpnontrp} we get that $R$ does not
satisfy the rank property.
\end{proof}

We
now discuss an example of a Legendrian knot $\Lambda\subset
\C\times \R$ for which $\MC_{\Lambda}$ does not satisfy the rank
property.

\begin{figure}[t]
\includegraphics[width=9cm]{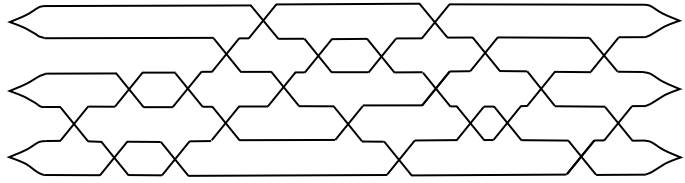}
\caption{The front projection of $\Lambda_{m(10_{132})}$ defined by
the braid words $4, 5, 3, 5, 3, 2, 4, 1, 3, 2, 4, 2, 5, 1, 3, 2, 4,
4, 3, 5, 4, 2$; cf  \protect\cite[Figure 2]{TCHOLKWMTBI}}
\label{monzeonthtw}
\end{figure}

Consider the Legendrian representative $\Lambda_{m(10_{132})}$ of
$m(10_{132})$ shown in \autoref{monzeonthtw}. We recall some facts
about its characteristic algebra $\MC_{\Lambda_{m(10_{132})}}
=\Z_2\langle x_1,\ldots,x_{25} \rangle/B$, using the notation of
\cite{TCHOLKWMTBI}, where $B$ is a two-sided ideal generated by
$\{\partial x_{1},\dots,\partial x_{25} \}$. Setting
$x:=1+x_{5}(x_{2} + x_{3})$ and $y:=x_{20}$, Sivek
\cite{TCHOLKWMTBI} showed that $xy=1$ and that the two-sided ideal
$I
\subset \mathcal{C}_{\Lambda_{m(10_{132})}}$ generated
by $1-yx$ coincides with $\MC_{\Lambda_{m(10_{132})}}$. Even more
can be said: taking $p:=x_{13} + x_{8}(x_{2}+x_{3})$ and $q:=x_{18}$,
we moreover have $p(1-yx)q=1$. This follows from the formulas
$$
\partial x_{22}=xq
\:\:\text{and}\:\:
\partial x_{23}=1+x_{11}x_{22}+pq
$$
in \cite{TCHOLKWMTBI}, taking into count that $x_{11}$ vanishes
in $\MC_{\Lambda_{m(10_{132})}}$ (as computed in
\cite[Section~2.2]{TCHOLKWMTBI}).
Hence, using \autoref{factnonrprop}, we see that
$\MC_{\Lambda_{m(10_{132})}}$
does not satisfy the rank property.

Note that it follows from \autoref{rkproptostrrankpropnatbnvintrghsk}
that $\MC_{\Lambda}\otimes \MC_{\Lambda'}^{\mathrm{op}}$ does not
satisfy the strong rank property for
$\Lambda=\Lambda_{m(10_{132})}$ either.

\begin{proof}[Proof of \autoref{hdbexwithoutthestrrnkpr}]
Assume that we are given a Legendrian knot $\widetilde{\Lambda}\subset
\C\times\R$ for which $\MC_{\widetilde{\Lambda}}$ does not satisfy the
rank property.
For example, $\widetilde{\Lambda}=\Lambda_{m(10_{132})}$.

We then use the spherical front spinning construction. For a given $n\in
\N$, consider $\Sigma_{S^{m_s}}\cdots\Sigma_{S^{m_1}}\widetilde{\Lambda}$,
where $\sum_{i=1}^{s}m_i+1=n$.
Observe that from \autoref{maininclinviig} it follows that there
is a homomorphism of unital algebras from $\MC_{\widetilde{\Lambda}}$ to
$\MC_{\Sigma_{S^{m_s}}\cdots\Sigma_{S^{m_1}}\widetilde{\Lambda}}$ given by
the composition
of homomorphisms

$$
\MC_{\widetilde{\Lambda}}\to
\MC_{\Sigma_{S^{m_1}}\widetilde{\Lambda}}\to\dots\to
\MC_{\Sigma_{S^{m_s}}\dots\Sigma_{S^{m_1}}\widetilde{\Lambda}}.
$$
Since $\MC_{\widetilde{\Lambda}}$ does not satisfy the
rank property, it follows from \autoref{rhomoibnrcondpres} that neither does
$\MC_{\Sigma_{S^{m_s}}\cdots\Sigma_{S^{m_1}}\widetilde{\Lambda}}$.
We define $\Lambda$ to be
$\Sigma_{S^{m_s}}\cdots\Sigma_{S^{m_1}}\widetilde{\Lambda}$.

We also observe that \autoref{rkproptostrrankpropnatbnvintrghsk} implies
that $\MC_{\Lambda}\otimes
\MC_{\Lambda'}^{\mathrm{op}}$ does not satisfy the strong rank property.
This finishes the proof.
\end{proof}

\begin{rem}
If there exists an exact Lagrangian cobordism $L$ from $\Lambda_{-}$
to $\Lambda_{+}$ and
$\MC_{\Lambda_+}$ does not satisfy the
rank property, then the same is true for $\MC_{\Lambda_{-}}$. Recall
that $L$ induces a unital DGA-morphism $\mathcal{A}(\Lambda_{+}) \to
\mathcal{A}(\Lambda_{-})$ and hence a unital algebra morphism
$\MC_{\Lambda_{+}}\to \MC_{\Lambda_{-}}$.
\autoref{rhomoibnrcondpres} now implies the claim.
\end{rem}

\begin{rem}
Assume that we are given a Legendrian submanifold $\Lambda\subset P\times
\R$ for which
$\MC_{\Lambda}$ does not satisfy the rank property.
It follows from
\autoref{rhomoibnrcondpres} that, even if we try to
reduce $\MC(\Lambda,\Lambda')$ to another bimodule by using a
homomorphism $\MC_{\Lambda}\to S$ (where $S$ is a unital ring), we will
necessarily obtain a free bimodule over a ring which does not satisfy the
rank property.
\vadjust{\goodbreak}
\end{rem}

\section{Proof of \texorpdfstring{\autoref{aprbpundalgprfbwhkn}}{Proposition 1.12}}

We start by showing that there must exist at least three Reeb chords on a
horizontally displaceable Legendrian submanifold whose Chekanov--Eliashberg
algebra is not acyclic, but whose characteristic algebra admits no
finite-dimensional representations. In particular, we prove that the bound
$|\MQ(\Lambda)|<3$ contradicts the hypothesis.

In the case when $|\MQ(\Lambda)|=1$, the differential must be trivial,
and hence the characteristic
algebra is free and obviously has a finite-dimensional
representation.

In the case when $|\MQ(\Lambda)|=2$, the characteristic algebra is
the quotient of the Chekanov--Eliashberg algebra $\langle
a,b\rangle$, $\ell(a) \le \ell(b)$, by the two-sided ideal generated
by a polynomial $p(a)$ (which is not a unit, by assumption). In other
words, the characteristic algebra is the free product $\F[b]  *
\F[a]/\langle p(a)\rangle$, which, hence, admits a unital projection
\[
\F[b]  * \F[a]/\langle p(a)\rangle \to \F[b]  * \F[a]/\langle p(a),b\rangle
= \F[a]/\langle p(a)\rangle
\]
to a non-zero commutative unital algebra. Finally, observe that
$\F[a]/\langle p(a)\rangle$ has a $1$--dimensional representation
$\F[a]/\langle p(a)\rangle\to \bigl(\F[a]/\langle p(a) \rangle\bigr)/I$, where $I$
is any maximal ideal of the commutative ring $\F[a]/\langle p(a)\rangle$.

\begin{rem}
At least algebraically, there are examples of free DGAs with
three generators for which the homology is not acyclic but for which the
characteristic algebra admits no finite-dimensional representations. Consider
the DGA generated by $\langle a,b,c\rangle$ over a field of characteristic
zero for which $\partial(a)=\partial(b)=0$ and $\partial(c)=1-(ab-ba)$. A
finite-dimensional unital representation of the characteristic algebra must
send $ab-ba$ to the identity. Since the trace of a commutator vanishes, while
the trace of the identity is non-zero, such a representation cannot exist.
\end{rem}

We
now prove \autoref{aprbpundalgprfbwhkn} in case (1). Let $c_{\mathrm{even}}$ and $c_{\mathrm{odd}}$ denote
the number of generators in even and odd degree, respectively. Using
\cite[Proposition 3.3]{NILSIRTNPO} we can express the Thurston--Bennequin
number as
\[
\operatorname{tb}(\Lambda)=(-1)^{(n-2)(n-1)/2}(c_{\mathrm{even}}-c_{\mathrm{odd}})=(-1)^{k+1}(c_{\mathrm{even}}-c_{\mathrm{odd}}).
\]
Combining this with the identity
\[
\operatorname{tb}(\Lambda)=(-1)^{k+1}\tfrac{1}{2}\chi(\Lambda)
\]
from \cite[Proposition 3.2(2)]{NILSIRTNPO}, we conclude that
\[
c_{\mathrm{even}}=\tfrac{1}{2}|\chi(\Lambda)|+c_{\mathrm{odd}}
\]
holds under the additional assumption that $\chi(\Lambda) \ge 0$. Moreover,
the equality $c_{\mathrm{odd}}=0$ implies that the differential is trivial
and, hence, that there is a canonical (graded) augmentation. We must
therefore have $c_{\mathrm{odd}} \ge 1$, from which it follows that
$$
|\mathcal Q(\Lambda)|=c_{\mathrm{odd}}+c_{\mathrm{even}}=
\tfrac{1}{2}|\chi(\Lambda)|+2c_{\mathrm{odd}}\geq
\tfrac{1}{2}|\chi(\Lambda)|+2.
$$

The proof
in case (2) is analogous. It suffices to consider the case
$\chi(\Lambda)< 0$, for which the above expressions yield
\[
c_{\mathrm{odd}}=\tfrac{1}{2}|\chi(\Lambda)|+c_{\mathrm{even}}.
\]
The assumption that the Maslov class is non-vanishing, ie that the
Chekanov--Eliashberg algebra admits an integer-valued grading, together
with the assumption that all generators are of non-negative degree, has
the following strong implication. Given that $c_{\mathrm{even}}=0$, the
Chekanov--Eliashberg algebra actually vanishes in degree zero, from which
the existence of an augmentation follows (using the assumption that the DGA
is not acyclic). We must thus have $c_{\mathrm{even}} \ge 1$, from which
we obtain the sought inequality
$$
|\mathcal Q(\Lambda)|=c_{\mathrm{odd}}+c_{\mathrm{even}}=
\tfrac{1}{2}|\chi(\Lambda)|+2c_{\mathrm{even}}\geq
\tfrac{1}{2}|\chi(\Lambda)|+2.
$$

\end{document}